\documentclass[12pt]{article}

\usepackage[cp1251]{inputenc}
\usepackage[english,russian]{babel}
\usepackage[tbtags]{amsmath}
\usepackage{amsfonts,amssymb,amsthm,cite}

\usepackage[active]{srcltx}

\usepackage{color}

\usepackage{mathrsfs}

\usepackage{graphicx}

\usepackage[width=160mm, left=20mm, top=10mm, height=245mm, paper=a4paper]{geometry}

\newtheorem{theorem}{\hskip\parindent \textsc{Theorem}}[section]
\newtheorem{lemma}{\hskip\parindent \textsc{Lemma}}[section]

\newtheorem{proposition}{\hskip\parindent \textsc{Proposition}}[section]

\def\loc{\text{\rm loc}}
\def\supp{\mathop{\rm supp}\nolimits}

\renewcommand{\refname}{Bibliography}

\begin{document}


\begin{center} 
{\bf
\large {On weighted Hardy inequality with two-dimensional rectangular operator -- extension of the E. Sawyer theorem}} 
\end{center}
\begin{center}
{\bf  {V. D. Stepanov\footnote{Computing Center of FEB RAS, Khabarovsk 680000, Russia; E-mail: stepanov@mi-ras.ru.} and E. P. Ushakova\footnote{V. A. Trapeznikov Institute of Control Sciences of RAS, Moscow 117997, Russia; E-mail: elenau@inbox.ru.}}}
\end{center}
\vskip0.2cm

{\small  {\bf Abstract:} A characterization is obtained for those pairs of weights $v$ and $w$ on $\mathbb{R}^2_+$, for which the two--dimensional rectangular integration operator is bounded from a weighted Lebesgue space $L^p_v(\mathbb{R}^2_+)$ to $L^q_w(\mathbb{R}^2_+)$ for $1<p\not= q<\infty$, which is an essential complement to E. Sawyer's result \cite{Saw1} given for $1<p\leq q<\infty$. Besides, we declare that the E. Sawyer theorem is actual if $p=q$ only, for $p<q$ the criterion is less complicated. The case $q<p$ is new.}



\parindent 15pt
\section{Introduction}

Let $n\in\mathbb{N}$. For Lebesgue measurable functions $f(y_1,\ldots,y_n)$ on $\mathbb{R}_+^n:=(0,\infty)^n$ the $n$--dimensional rectangular integration operator $I_n$ is given by the formula
\begin{equation*}
I_nf(x_1,\ldots,x_n)\colon=\int_0^{x_1}\ldots
\int_0^{x_n}f(y_1,\ldots,y_n)\,{d}y_1\ldots dy_n\qquad (x_1,\ldots,x_n>0).
\end{equation*}
The dual transformation $I^\ast_n$ has the form
\begin{equation*}\label{n1*}
I^\ast_nf(x_1,\ldots,x_n)\colon=\int_{x_1}^\infty\ldots
\int_{x_n}^\infty f(y_1,\ldots,y_n)\,{d}y_1\ldots dy_n\qquad (x_1,\ldots,x_n>0).
\end{equation*}

Let $1<p,q<\infty$ and $v,w\ge 0$ be weight functions on $\mathbb{R}_+^n$.
Consider Hardy's inequality
\begin{equation}\label{n2}
\biggl(\int_{\mathbb{R}^n_+}\bigl(I_nf\bigr)^q
w\biggr)^{\frac{1}{q}}\le C_n
\biggl(\int_{\mathbb{R}^n_+}f^pv\biggr)^{\frac{1}{p}}\qquad (f\ge 0)
\end{equation}
on the cone of non--negative functions in weighted Lebesgue space $L^p_v(\mathbb{R}^n_+)$. The constant $C_n>0$ in \eqref{n2} is assumed to be the least possible and independent of $f$. For a fixed weight $v$ and a parameter $p>1$ the space $L^p_v(\mathbb{R}^n_+)$ consists of all measurable on $\mathbb{R}_+^n$ functions $f$ such that $\int_{\mathbb{R}_+^n}|f|^pv<\infty$.

The problem of characterizing the inequality \eqref{n2} is well known and has been considered by many authors (see \cite{Barza, KMP1, Mes, PU, Saw1, W1, B5} and references therein).  The one--dimensional case of this inequality has been completely studied (see \cite{Maz, KMP, KPS, PSU}).  However, for $n>1$ difficulties arise,
preventing characterizing \eqref{n2} without additional restrictions on $v$ and $w$. Nevertheless, E. Sawyer's result is well known for arbitrary $ v, w $ in the case $1<p\le
q<\infty$. To formulate it we denote $p':=p/(p-1)$ and $\sigma:=v^{1-p'}$.
\begin{theorem}{\rm\cite[Theorem 1A]{Saw1}}\label{ES}
Let $n=2$ and $1<p\le
q<\infty.$ The inequality \eqref{n2} holds for all measurable non-negative functions $f$ on $\mathbb{R}^2_+$
if and only if
\begin{equation}\label{n3'}
A_{1}:=A_{1}(p,q):=\sup_{(t_1,t_2)\in\mathbb{R}_+^2}\bigl[I^\ast_2 w(t_1,t_2)\bigr]^{\frac{1}{q}}
\bigl[I_2 \sigma(t_1,t_2)\bigr]^{\frac{1}{p'}}<\infty,\end{equation}\begin{equation}
\label{n4'}A_{2}:=A_{2}(p,q):=\sup_{(t_1,t_2)\in\mathbb{R}_+^2}\biggl(\int_0^{t_1}\int_0^{t_2}
\bigl(I_2\sigma\bigr)^q w\biggr)^{\frac{1}{q}}
\bigl[I_2 \sigma(t_1,t_2)\bigr]^{-\frac{1}{p}}<\infty,\end{equation}
\begin{equation}
\label{n5'}A_{3}:=A_{3}(p,q):=\sup_{(t_1,t_2)\in\mathbb{R}_+^2}\biggl(\int_{t_1}^\infty\int_{t_2}^\infty
\bigl(I_2^\ast w\bigr)^{p'}\sigma\biggr)^{\frac{1}{p'}}
\bigl[I_2^\ast w(t_1,t_2)\bigr]^{-\frac{1}{q'}}<\infty,
\end{equation}
{and $C_2\approx A_1+A_2+A_3$ with equivalence constants depending on parameters $p$ and $q$}.
\end{theorem}
Note that in one--dimensional case the analogs of the conditions  \eqref{n3'}--\eqref{n5'} are equivalent to each other \cite{GKPW}. For $n=2$ this, generally speaking, is not true. Moreover, as shown in \cite[\S\,4]{Saw1} for $p=q=2$, no two of the conditions \eqref{n3'}--\eqref{n5'} guarantee \eqref{n2}. However, the construction of the second counterexample in \cite[\S\,4]{Saw1} is unexpandable to the case $p<q.$

The purpose of this paper is to obtain new criteria for the fulfillment of Hardy's inequality \eqref{n2} for $n=2$ and $1<p\not=q<\infty$. The solution to this problem is contained in Theorem \ref{aga} (see \S\,2). In \S~3 an alternative sufficient condition is found for $ v $ and $ w $, when \eqref{n2} is true for $n=2$ and $1<q<p<\infty$. Recall that the criterion for \eqref{n2} when $n=2$ and $1<p\le q<\infty,$ established in  \cite{Saw1}, is that the sum of three independent functionals is bounded (see Theorem \ref{ES}). It is proven in Theorem \ref {aga} that for $ 1 <p \not = q <\infty $ the inequality \eqref {n2} is characterized by only one functional.

Analogs of Theorem \ref{aga} are also valid for the dual operator $I_2^\ast $ and mixed Hardy operators (see \cite[Remark 1]{Saw1} for details).

In \S \,4, for completeness, we present known results about the operator $ I_n $ for arbitrary $ n $, provided that at least one of the two weight functions in \eqref {n2} is factorizable, that is, can be represented as a product of $n$ one--dimensional functions.

Since $A_1\leq C_2$, we may and shall assume that $I_2\sigma(x,y)<\infty$ and $I_2^\ast w(x,y)<\infty$ for any $(x,y)\in \mathbb{R}_+^2.$ In particular, $\sigma, w \in L^1_{\loc}(\mathbb{R}_+^2).$

Throughout the work, the notation of the form $ \Phi \lesssim \Psi $ means that the relation $ \Phi \le c \Psi $ holds with some constant $ c> 0 $, independent of $ \Phi $ and $ \Psi $. We write $\Phi \approx \Psi $ in the case of $ \Phi\lesssim \Psi \lesssim \Phi $.
The symbols $ \mathbb Z $ and $\mathbb N $ are used for denoting the sets of integers and natural numbers, respectively. The characteristic function of the subset $ E\subset\mathbb {R}^n_+$ is denoted by $\chi_E $. Symbols $: = $ and
$ =: $ are used to define new values.

\section{Main result}

Denote
$$
\alpha(p,q):=\frac{p^2(q-1)}{q-p},\qquad p<q;
$$
$$
\beta(p,q):=\frac{2^{q+1}}{2^{\frac{q}{r}}-1}\cdot
\begin{cases}2^{\frac{q}{p}-\frac{q}{r}}, & \frac{r}{p}\ge 1,\\ 1, &\frac{r}{p}<1,
\end{cases} \qquad q<p,
$$
where ${1}/{r}:={1}/{q}-{1}/{p}$; $A:=A_1,$
\begin{multline*}B:=B_1:=B_1(p,q):=\biggl(
\int_{\mathbb{R}^2_+} d_y\bigl[I_2\sigma(x,y)\bigr]^{\frac{r}{p'}}\,d_x\Bigl(-\bigl[I_2^\ast w(x,y)\bigr]^{\frac{r}{q}}\Bigr)\biggr)^{\frac{1}{r}}\\=\biggl(\int_{\mathbb{R}_+^2}\bigl[I_2 \sigma(x,y)\bigr]^{\frac{r}{p'}}\,d_x\,d_y\bigl[I^\ast_2 w(x,y)\bigr]^{\frac{r}{q}}\biggr)^{\frac{1}{r}}=
\biggl(\int_{\mathbb{R}_+^2}\bigl[I^\ast_2 w(x,y)\bigr]^{\frac{r}{q}}\,d_x\,d_y\bigl[I_2 \sigma(x,y)\bigr]^{\frac{r}{p'}}\biggr)^{\frac{1}{r}},\end{multline*} where the last two equalities follow by integration by parts; also
\begin{align*}
B_{2}:=B_2(p,q):=&\Biggl(\int_{\mathbb{R}_+^2}\bigl[I_2 \sigma(x,y)\bigr]^{-\frac{r}{p}}\,d_x\,d_y\biggl(\int_0^{x}\int_0^{y}(I_2\sigma)^q w\biggr)^{\frac{r}{q}}\Biggr)^{\frac{1}{r}},\\ B_{3}:=B_3(p,q):=&\Biggl(\int_{\mathbb{R}_+^2}\bigl[I_2 ^\ast w(x,y)\bigr]^{-\frac{r}{q'}}\,d_x\,d_y\biggl(\int_{x}^\infty\int_{y}^\infty(I_2^\ast w)^{p'} \sigma\biggr)^{\frac{r}{p'}}\Biggr)^{\frac{1}{r}}.\end{align*}
 Notice that
\begin{equation}\label{AB}
\lim\limits_{q\uparrow p}B_i(p,q)=A_i(p,p),\quad i=1,2,3.
\end{equation}

Let us recall the result we need in what follows from the work \cite{GHS}.
\begin{proposition}\label{propGHS}{\rm \cite[Proposition 2.1]{GHS}}
Let $0<\gamma<\infty$ and let $\{a_k\}_{k\in\mathbb{Z}}$, $\{\rho_k\}_{k\in\mathbb{Z}}$, $\{\tau_k\}_{k\in\mathbb{Z}}$ be non-negative sequences.\\
{\rm (a)} If $\rho:=\inf_{k\in\mathbb{Z}}{\rho_{k+1}}/{\rho_k}>1$ then $$\sum_{k\in\mathbb{Z}}\Bigl(\sum_{m\ge k}a_m\Bigr)^\gamma\rho_k^\gamma\le \sum_{m\in\mathbb{Z}}a_m^\gamma\rho_m^\gamma\cdot\begin{cases}\frac{\rho^\gamma}{\rho^\gamma-1},& 0<\gamma\le 1,\\ \frac{\rho^\gamma}{(\rho^{\gamma'-1}-1)^{\gamma-1}(\rho^{\gamma-1}-1)}, &\gamma>1.\end{cases}$$\\
{\rm (b)} If $\tau:=\sup_{k\in\mathbb{Z}}{\tau_{k+1}}/{\tau_k}<1$ then
$$\sum_{k\in\mathbb{Z}}\Bigl(\sum_{m\le k}a_m\Bigr)^\gamma\tau_k^\gamma\le \sum_{m\in\mathbb{Z}}a_m^\gamma\tau_m^\gamma\cdot\begin{cases}\frac{\tau^{-\gamma}}{\tau^{-\gamma}-1},& 0<\gamma\le 1,\\ \frac{\tau^{-\gamma}}{(\tau^{1-\gamma'}-1)^{\gamma-1}(\tau^{1-\gamma}-1)}, &\gamma>1.\end{cases}$$
\end{proposition}
We start with some auxiliary technical statements.
\begin{lemma}\label{lem1} Let $0\le a<b<\infty$ and $0\le c<d<\infty$.
If $1<p<q<\infty$ then \begin{equation*}
\mathbf{V}_{(a,b)\times(c,d)}:=\int_a^b\int_c^d w(x,y)\Bigl(\int_a^x\int_c^y \sigma\Bigr)^q\,dy\,dx\le \alpha(p,q)
\Bigl(\int_a^b\int_c^d\sigma\Bigr)^{\frac{q}{p}}A^q.
\end{equation*}
For $1<q<p<\infty$ the following inequality holds:
\begin{multline*}
\mathbf{V}_{(a,b)\times(c,d)}\le \beta(p,q)\Bigl(\int_a^b\int_c^d\sigma\Bigr)^{\frac{q}{p}}\\\times\biggl[\int_a^b\int_c^d \chi_{\supp\,w}(x,y)d_y\bigl[I_2\sigma(x,y)\bigr]^{\frac{r}{p'}}\,d_x\Bigl(-\bigl[I_2^\ast w(x,y)\bigr]^{\frac{r}{q}}\Bigr)\biggr]^{\frac{q}{r}}.
\end{multline*}
\end{lemma}\begin{proof}[Proof] Assume $1<p<q<\infty$ and write
\begin{align*}\mathbf{V}_{(a,b)\times(c,d)}
=&\int_a^b\int_c^d \Bigl(\int_a^x\int_c^y \sigma\Bigr)^q\,d_y\biggl[-\int_y^d w(x,t)\,dt\biggr]dx\\=&q
\int_a^b\int_c^d \Bigl(\int_a^x\int_c^y \sigma\Bigr)^{q-1}\Bigl(\int_a^x\sigma(s,y)\,ds\Bigr)\Bigl(\int_y^d w(x,t)\,dt\Bigr)\,dy\,dx\\=&q
\int_c^d\int_a^b \Bigl(\int_a^x\int_c^y \sigma\Bigr)^{q-1}\Bigl(\int_a^x\sigma(s,y)\,ds\Bigr)\,d_x\Bigl[-\int_x^b\int_y^d w\Bigr]\, dy\\=&q
\int_a^b\int_c^d \biggl\{(q-1)\Bigl(\int_a^x\int_c^y \sigma\Bigr)^{q-2}\Bigl(\int_a^x\sigma(s,y)\,ds\Bigr)\Bigl(\int_c^y\sigma(x,t)\,dt\Bigr)\\&+ \Bigl(\int_a^x\int_c^y \sigma\Bigr)^{q-1}\sigma(x,y)\biggr\}\Bigl(\int_x^b\int_y^d w\Bigr)\,dx\, dy.\end{align*} Then
\begin{multline*}\mathbf{V}_{(a,b)\times(c,d)}\le
qA^q\int_a^b\int_c^d \biggl\{(q-1)\Bigl(\int_a^x\int_c^y \sigma\Bigr)^{\frac{q}{p}-2}\Bigl(\int_a^x\sigma(s,y)\,ds\Bigr)\Bigl(\int_c^y\sigma(x,t)\,dt\Bigr)\\+ \Bigl(\int_a^x\int_c^y \sigma\Bigr)^{\frac{q}{p}-1}\sigma(x,y)\biggr\}\,dx\, dy.\end{multline*}
The assertion of the lemma for the case $ p <q $ follows from the chain of inequalities:
\begin{align*}&
q\int_a^b\int_c^d \biggl\{(q-1)\Bigl(\int_a^x\int_c^y \sigma\Bigr)^{\frac{q}{p}-2}\Bigl(\int_a^x\sigma(s,y)\,ds\Bigr)\Bigl(\int_c^y\sigma(x,t)\,dt\Bigr)\\&+ \Bigl(\int_a^x\int_c^y \sigma\Bigr)^{\frac{q}{p}-1}\sigma(x,y)\biggr\}\,dx\, dy\\=&{
p\int_a^b\int_c^d \biggl\{\frac{q}{p}\Bigl(\frac{q}{p}-1+\frac{q}{p'}\Bigr)\Bigl(\int_a^x\int_c^y \sigma\Bigr)^{\frac{q}{p}-2}\Bigl(\int_a^x\sigma(s,y)\,ds\Bigr)\Bigl(\int_c^y\sigma(x,t)\,dt\Bigr)}\\&+\frac{q}{p} \Bigl(\int_a^x\int_c^y \sigma\Bigr)^{\frac{q}{p}-1}\sigma(x,y)\biggr\}\,dx \,dy\\\le&{
p\int_a^b\int_c^d \biggl\{\frac{q}{p}\Bigl(\frac{q}{p}-1\Bigr)\Bigl(\int_a^x\int_c^y \sigma\Bigr)^{\frac{q}{p}-2}\Bigl(\int_a^x\sigma(s,y)\,ds\Bigr)\Bigl(\int_c^y\sigma(x,t)\,dt\Bigr)}\\&+\frac{q}{p} \Bigl(\int_a^x\int_c^y \sigma\Bigr)^{\frac{q}{p}-1}\sigma(x,y)\biggr\}\,dx \,dy\\&+{
\frac{p^2q^2}{p'q(q-p)}\int_a^b\int_c^d \biggl\{\frac{q}{p}\Bigl(\frac{q}{p}-1\Bigr)\Bigl(\int_a^x\int_c^y \sigma\Bigr)^{\frac{q}{p}-2}\Bigl(\int_a^x\sigma(s,y)\,ds\Bigr)\Bigl(\int_c^y\sigma(x,t)\,dt\Bigr)}\\&+\frac{q}{p} \Bigl(\int_a^x\int_c^y \sigma\Bigr)^{\frac{q}{p}-1}\sigma(x,y)\biggr\}\,dx\, dy\\=&\biggl[p+
\frac{pq(p-1)}{q-p}\biggr]\int_a^b\int_c^d \biggl\{\frac{q}{p}\Bigl(\frac{q}{p}-1\Bigr)\Bigl(\int_a^x\int_c^y \sigma\Bigr)^{\frac{q}{p}-2}\Bigl(\int_a^x\sigma(s,y)\,ds\Bigr)\Bigl(\int_c^y\sigma(x,t)\,dt\Bigr)\\&+\frac{q}{p} \Bigl(\int_a^x\int_c^y \sigma\Bigr)^{\frac{q}{p}-1}\sigma(x,y)\biggr\}\,dx\, dy\\=&
\alpha(p,q)\int_a^b\int_c^d \biggl\{\frac{q}{p}\Bigl(\frac{q}{p}-1\Bigr)\Bigl(\int_a^x\int_c^y \sigma\Bigr)^{\frac{q}{p}-2}\Bigl(\int_a^x\sigma(s,y)\,ds\Bigr)\Bigl(\int_c^y\sigma(x,t)\,dt\Bigr)\\&+\frac{q}{p} \Bigl(\int_a^x\int_c^y \sigma\Bigr)^{\frac{q}{p}-1}\sigma(x,y)\biggr\}\,dx\, dy=\alpha(p,q)\Bigl(\int_a^b\int_c^d\sigma\Bigr)^{\frac{q}{p}}.
\end{align*}

Now suppose that $q<p$.
By analogy with the proof of \cite[Theorem 1A]{Saw1} we define the domains
$$
\omega_k:=\Bigl\{(x,y)\in(a,b)\times(c,d): \int_a^x\int_c^y \sigma>2^k\Bigr\},\qquad -\infty<k\le K_\sigma.
$$
The restriction $ K_\sigma <\infty $ follows from the condition \cite[(1.6)]{Saw1}, which is necessary for any relations between $ p $ and $ q $. Then
\begin{multline*}
\mathbf{V}_{(a,b)\times(c,d)}=\sum_{k\le K_\sigma}\int_{\omega_k\setminus\omega_{k+1}}w(x,y)\Bigl(\int_a^x\int_c^y \sigma\Bigr)^q\,dy\,dx\\
\le
2^q\sum_{k\le K_\sigma}2^{kq}\bigl|\omega_k\setminus\omega_{k+1}\bigr|_w\le
2^q\sum_{k\le K_\sigma}2^{kq}\bigl|\omega_k\bigr|_w,
\end{multline*} where $\bigl|\omega_k\bigr|_w:=\int_{\omega_k}w$.
Denote
$\alpha_k:=\inf\bigl\{x:a\le x\colon w\chi_{\omega_k}(x,y)>0\bigr\}$,
$\beta_k:=\inf\bigl\{y:c\le y\colon w\chi_{\omega_k}(x,y)>0\bigr\}.$ Observe that $\alpha_k>a$ and $\beta_k>c$ and write
\begin{multline*}
\bigl|\omega_k\bigr|_w=\int_{\alpha_k}^b\int_{\beta_k}^d w\chi_{\omega_k}=\biggl(
\int_{\beta_k}^d\,d_y\biggl[-\Bigl(\int_y^d\int_{\alpha_k}^b w\chi_{\omega_k}\Bigr)^{\frac{r}{q}}\biggr]\biggr)^{\frac{q}{r}}\\=\biggl(
\int_{\beta_k}^d\,d_y\biggl\{-\int_{\alpha_k}^b d_x\biggl[-\Bigl(\int_{x}^b\int_y^d w\chi_{\omega_k}\Bigr)^{\frac{r}{q}}\biggr]\biggr\}\biggr)^{\frac{q}{r}}.
\end{multline*}
Since $\Bigl[-\Bigl(\int_{x}^b\int_y^d w\chi_{\omega_k}\Bigr)^{{r}/{q}}\Bigr]'_x=0$ out of the set
$\omega_k\cap{\rm supp}\,w$
for each fixed $y\ge\beta_k$ and, analogously, $\Bigl[-\Bigl(\int_{x}^b\int_y^d w\chi_{\omega_k}\Bigr)^{{r}/{q}}\Bigr]'_y=0$ outside $\omega_k\cap{\rm supp}\,w$ for all $x\ge\alpha_k$, then
\begin{equation*}\bigl|\omega_k\bigr|_w=\biggl(
\int_{\beta_k}^d\int_{\alpha_k}^b\,d_xd_y\Bigl(\int_{x}^b\int_y^d w\chi_{\omega_k}\Bigr)^{\frac{r}{q}}\biggr)^{\frac{q}{r}}=\biggl(
\int_{\omega_k}\chi_{\textrm{supp}\,w}(x,y)\,d_xd_y\Bigl(\int_{x}^b\int_y^d w\Bigr)^{\frac{r}{q}}\biggr)^{\frac{q}{r}}.\end{equation*} Due to the choice of $ \omega_k $,
\begin{multline*}2^{{kq}}\bigl|\omega_k\bigr|_w=2^{{kq}}\biggl(
\int_{\omega_k}\chi_{\textrm{supp}\,w}(x,y)\,d_xd_y\Bigl(\int_x^b\int_y^d w\Bigr)^{\frac{r}{q}}\biggr)^{\frac{q}{r}}\\\le
2^{\frac{kq}{r}}\biggl(
\int_{\omega_k} \chi_{\textrm{supp}\,w}(x,y)\Bigl(\int_a^x\int_c^y\sigma\Bigr)^{{r-1}}\,d_xd_y\Bigl(\int_x^b\int_y^d w\Bigr)^{\frac{r}{q}}\biggr)^{\frac{q}{r}}.\end{multline*} It follows from Proposition \ref{propGHS}(a) with $\rho=2$ and $\gamma=q/r<1$ that \begin{multline*}\sum_{k\le K_\sigma}2^{\frac{kq}{r}}
\biggl(
\int_{\omega_k} \chi_{\textrm{supp}\,w}(x,y)\Bigl(\int_a^x\int_c^y\sigma\Bigr)^{r-1}\,d_xd_y\Bigl(\int_x^b\int_y^d w\Bigr)^{\frac{r}{q}}\biggr)^{\frac{q}{r}}\\=
\sum_{k\le K_\sigma}2^{\frac{kq}{r}}
\biggl(\sum_{m\ge k}
\int_{\omega_m\setminus\omega_{m+1}}\chi_{\textrm{supp}\,w}(x,y) \Bigl(\int_a^x\int_c^y\sigma\Bigr)^{r-1}\,d_xd_y\Bigl(\int_x^b\int_y^d w\Bigr)^{\frac{r}{q}}\biggr)^{\frac{q}{r}}\\\le\frac{2^{\frac{q}{r}}}{2^{\frac{q}{r}}-1}
\sum_{k\le K_\sigma}2^{\frac{kq}{r}}
\biggl(\int_{\omega_k\setminus\omega_{k+1}}\chi_{\textrm{supp}\,w}(x,y) \Bigl(\int_a^x\int_c^y\sigma\Bigr)^{r-1}\,d_xd_y\Bigl(\int_x^b\int_y^d w\Bigr)^{\frac{r}{q}}\biggr)^{\frac{q}{r}}\\\le\frac{\beta(p,q)}{2^{q+\frac{q}{p}}}
\sum_{k\le K_\sigma}2^{\frac{kq}{p}}
\biggl(\int_{\omega_k\setminus\omega_{k+1}}\chi_{\textrm{supp}\,w}(x,y) \Bigl(\int_a^x\int_c^y\sigma\Bigr)^{\frac{r}{p'}}\,d_xd_y\Bigl(\int_x^b\int_y^d w\Bigr)^{\frac{r}{q}}\biggr)^{\frac{q}{r}}.\end{multline*}
Using H\"older's inequality with exponents $ r / q $ and $ p / q $, we obtain
\begin{multline*}
\sum_{k\le K_\sigma}2^{\frac{kq}{p}}
\biggl(\int_{\omega_k\setminus\omega_{k+1}}\chi_{\textrm{supp}\,w}(x,y) \Bigl(\int_a^x\int_c^y\sigma\Bigr)^{\frac{r}{p'}}\,d_xd_y\Bigl(\int_x^b\int_y^d w\Bigr)^{\frac{r}{q}}\biggr)^{\frac{q}{r}}\\\le2^{\frac{q}{p}}2^{\frac{qK_\sigma}{p}}
\biggl(\sum_{k\le K_\sigma}\int_{\omega_k\setminus\omega_{k+1}} \chi_{\textrm{supp}\,w}(x,y)\Bigl(\int_a^x\int_c^y\sigma\Bigr)^{\frac{r}{p'}}\,d_xd_y\Bigl(\int_x^b\int_y^d w\Bigr)^{\frac{r}{q}}\biggr)^{\frac{q}{r}}\\\le 2^{\frac{q}{p}} \Bigl(\int_a^b\int_c^d\sigma\Bigr)^{\frac{q}{p}}
\biggl(\int_a^b\int_c^d \chi_{\textrm{supp}\,w}(x,y)\Bigl(\int_a^x\int_c^y\sigma\Bigr)^{\frac{r}{p'}}\,d_xd_y\Bigl(\int_x^b\int_y^d w\Bigr)^{\frac{r}{q}}\biggr)^{\frac{q}{r}}
.\end{multline*} Since $ r / q> 1 $ and $ r / p '> 1 $, then integrating by parts over the variable $y$ yields
\begin{multline*}
\int_a^b\int_c^d \chi_{\textrm{supp}\,w}(x,y)\Bigl(\int_a^x\int_c^y\sigma\Bigr)^{\frac{r}{p'}}\,d_yd_x\Bigl(\int_x^b\int_y^d w\Bigr)^{\frac{r}{q}}\\=\int_a^b\int_c^d \chi_{\textrm{supp}\,w}(x,y)d_y\Bigl(\int_a^x\int_c^y\sigma\Bigr)^{\frac{r}{p'}}\,d_x\biggl[-\Bigl(\int_x^b\int_y^d w\Bigr)^{\frac{r}{q}}\biggr]\\\le\int_a^b\int_c^d \chi_{\textrm{supp}\,w}(x,y)d_y\bigl[I_2\sigma(x,y)\bigr]^{\frac{r}{p'}}\,d_x\Bigl(-\bigl[I_2^\ast w(x,y)\bigr]^{\frac{r}{q}}\Bigr).\qedhere\end{multline*}
\end{proof} A similar statement holds with the (inner) integral of $ w $.
\begin{lemma}\label{lem2}
Let $0\le a<b<\infty$ and $0\le c<d<\infty$.
If $1<p<q<\infty$ then \begin{equation*}\mathbf{W}_{(a,b)\times(c,d)}:=\int_a^b\int_c^d \sigma(x,y)\Bigl(\int_x^b\int_y^d w\Bigr)^{p'}\,dy\,dx\le \alpha(q',p')
\Bigl(\int_a^b\int_c^dw\Bigr)^{\frac{p'}{q'}}A^{p'}.\end{equation*} In the case $1<q<p<\infty$
\begin{multline*}
\mathbf{W}_{(a,b)\times(c,d)}\le \beta(q',p')
\Bigl(\int_a^b\int_c^dw\Bigr)^{\frac{p'}{q'}}\\\times\biggl[\int_a^b\int_c^d \chi_{\supp\sigma}(x,y)d_y\bigl[I_2\sigma(x,y)\bigr]^{\frac{r}{p'}}\,d_x\Bigl(-\bigl[I_2^\ast w(x,y)\bigr]^{\frac{r}{q}}\Bigr)\biggr]^{\frac{p'}{r}}.
\end{multline*}
\end{lemma}
Introduce notations: $\alpha:=\alpha(p,q)$, $\beta:=\beta(p,q)$, $\alpha':=\alpha(q',p')$, $\beta':=\beta(q',p')$, $$\mathbb{C}_{\alpha,\alpha'}:=3^{3q}\Bigl[\Bigl(\frac{2}{3}\Bigr)^{q}\max\Bigl\{\alpha,2q(q')^{\frac{q}{p'}}\Bigr\}\Bigl(\frac{2^{p-1}}{2^{p-1}-1}\Bigr)^{\frac{q}{p}}+3^{\frac{1}{p}}(\alpha')^{\frac{1}{p'}}\Bigl(\frac{3^{q-1}}{3^{q-1}-1}\Bigr)^{\frac{1}{q'}}\Bigr],
$$ $$
\mathbf{C}_{\beta,\beta'}:=3^{3q}\Bigl[\Bigl(\frac{2}{3}\Bigr)^{q}\max\Bigl\{\beta,2q(p')^{q-1}\Bigl(\frac{q}{r}\Bigr)^{\frac{q}{r}}\Bigr\}
\Bigl(\frac{2^{p-1}}{2^{p-1}-1}\Bigr)^{\frac{q}{p}}\\+3(\beta')^{\frac{1}{p'}}\Bigl(\frac{3^{q-1}}{3^{q-1}-1}\Bigr)^{\frac{1}{q'}}\Bigr].$$

The main result of the work is the following statement.

\begin{theorem}\label{aga}
Let $1<p\not=q<\infty$. If $p<q$ then the inequality
\begin{equation}\label{aux}
\biggl(\int_{\mathbb{R}^2_+}\bigl(I_2f\bigr)^q
w\biggr)^{\frac{1}{q}}\le C_2
\biggl(\int_{\mathbb{R}^2_+}f^pv\biggr)^{\frac{1}{p}}\qquad(f\ge 0)
\end{equation}
holds if and only if $A<\infty.$ Besides,
$$
A\le C_2\le\mathbb{C}_{\alpha,\alpha'}\,A.
$$
In the case $q<p$ the inequality \eqref{aux} is true if and only if
$B<\infty.$ Moreover, \begin{equation*}2^{-\frac{1}{p'}}\Bigl(\frac{q}{r}\Bigr)^{\frac{1}{q}}\Bigl(\frac{p'}{r}\Bigr)^{\frac{1}{p'}}B\le  C_2\le\mathbf{C}_{\beta,\beta'}\,B.
\end{equation*}
\end{theorem}
\begin{proof}[Proof]
(\textit{Sufficiency}) Similarly to how it was done in E. Sawyer's paper \cite{Saw1} for the case $1<p\le q<\infty$, we show that the conditions of the theorem are sufficient,
limiting ourselves to proving the inequality \eqref{aux} on the subclass
$M\subset L^p_v(\mathbb{R}^2_+)$ of all functions $f\ge
0$ bounded on $\mathbb{R}^2_+$ with compact supports contained in the set $\{I_2\sigma>0\}$. Then the inequality \eqref{aux} for arbitrary $ 0 \leq f \in L^p_v (\mathbb {R}^2_+) $ follows by the standard arguments.

Suppose $ A <\infty $ for $ p <q $ (or $ B <\infty $ in the case of $ q <p $) and fix $ f \in M $. By analogy with the proof of \cite[Theorem 1A]{Saw1}, we define the domains
$$
\Omega_k\colon=\left\{I_2f>3^{k}\right\},\qquad k\in\mathbb{Z}.
$$
Then, by our assumptions on $f$, there exists $K\in\mathbb{Z}$ such that $\Omega_k\not=\varnothing$ for $k\leq K,$ $\Omega_k=\varnothing$ for $k>K$,  $\bigcup_{k\in\mathbb{Z}}\Omega_{k}=\mathbb{R}^2_+$  and
\begin{equation*}\label{sv}
3^{k}< I_2f(x,y)\leq 3^{k+1},\qquad k\leq K,\qquad
(x,y)\in\left(\Omega_{k}\setminus\Omega_{k+1}\right).
\end{equation*}

{\setlength{\unitlength}{0.0105in}{
\begin{picture}(350,350)
\put(10,10){\vector(0,1){320}}
\put(10,10){\vector(1,0){380}}

{\linethickness{.36mm}
\qbezier(40,270)(80,45)(300,25)
\qbezier(70,320)(130,60)(380,35)}

\put(60,0){\small $x_1^k$}
\put(120,0){\small $x_2^k$}
\put(200,0){\small $x_3^k$}

\put(-3,43){\small $y_3^k$}
\put(-3,94){\small $y_2^k$}
\put(-3,182){\small $y_1^k$}

\put(168,150){$\widetilde{\bf S}_{\boldsymbol{2}}^{\boldsymbol{k}}$}
\put(170,73){$\widetilde{\bf R}_{\boldsymbol{2}}^{\boldsymbol{k}}$}
\put(280,180){${\bf T}_{\boldsymbol{2}}^{\boldsymbol{k}}$}

{\linethickness{.0001mm}
\put(124,47){\line(1,1){50}}\put(132,47){\line(1,1){50}}\put(140,47){\line(1,1){50}}\put(148,47){\line(1,1){50}}\put(156,47){\line(1,1){48}}\put(164,47){\line(1,1){40}}
\put(172,47){\line(1,1){32}}\put(180,47){\line(1,1){25}}
\put(188,47){\line(1,1){16}}\put(195,47){\line(1,1){10}}

\put(124,55){\line(1,1){42}}\put(124,63){\line(1,1){34}}\put(124,71){\line(1,1){26}}\put(124,79){\line(1,1){18}}\put(124,87){\line(1,1){10}} }

{\linethickness{.0001mm}
\put(122,184){\line(1,-1){80}}\put(130,184){\line(1,-1){72}}\put(138,184){\line(1,-1){64}}\put(146,184){\line(1,-1){56}}\put(154,184){\line(1,-1){48}}\put(162,184){\line(1,-1){40}}
\put(170,184){\line(1,-1){32}}\put(178,184){\line(1,-1){25}}
\put(186,184){\line(1,-1){16}}\put(193,184){\line(1,-1){10}}

\put(124,174){\line(1,-1){77}}\put(124,166){\line(1,-1){68}}\put(124,158){\line(1,-1){60}}\put(124,150){\line(1,-1){52}}\put(124,142){\line(1,-1){44}}\put(124,134){\line(1,-1){36}}
\put(124,126){\line(1,-1){28}}\put(124,118){\line(1,-1){20}}
\put(124,110){\line(1,-1){12}} }

{\linethickness{.0001mm}
\put(205,107){\line(1,-1){10}}\put(205,117){\line(1,-1){20}}\put(205,127){\line(1,-1){30}}\put(205,137){\line(1,-1){40}}\put(205,147){\line(1,-1){50}}\put(205,157){\line(1,-1){60}}
\put(205,167){\line(1,-1){70}}\put(205,177){\line(1,-1){80}}
\put(205,187){\line(1,-1){90}}\put(205,197){\line(1,-1){100}}
\put(205,207){\line(1,-1){110}}\put(205,217){\line(1,-1){120}}\put(205,227){\line(1,-1){130}}\put(205,237){\line(1,-1){140}}
\put(205,247){\line(1,-1){150}}\put(205,257){\line(1,-1){160}}
\put(205,267){\line(1,-1){170}}\put(205,277){\line(1,-1){170}}\put(205,287){\line(1,-1){170}}\put(205,297){\line(1,-1){170}}
\put(205,307){\line(1,-1){170}}\put(205,317){\line(1,-1){170}}
\put(205,327){\line(1,-1){170}}
}



\put(65,10){\line(0,1){315}}

\put(10,184){\line(1,0){114}}
\multiput(124,184)(8.5,0){10}{\line(1,0){5}}

\put(124,10){\line(0,1){175}}
\multiput(124,185)(0,8){18}{\line(0,1){5}}

\put(10,97.3){\line(1,0){198}}

\put(204,10){\line(0,1){174}}
\multiput(204,184)(0,8){18}{\line(0,1){5}}

\put(10,47){\line(1,0){300}}

\multiput(310,47)(8.5,0){7}{\line(1,0){5}}

\multiput(210,97.3)(8.5,0){20}{\line(1,0){5}}

\put(385,38){\small $\partial\Omega_{k+1}$}
\put(310,25){\small $\partial\Omega_{k}$}

\put(203,-15){\scriptsize\rm Fig. 1}
\end{picture}}}\vspace{8mm}

We can write down that \begin{equation*}\label{10-mult}\int_{\mathbb{R}^2_+}
(I_2f)^qw=\sum_{k\le K-2}\int_{\Omega_{k+2}\setminus\Omega_{k+3}}
(I_2f)^qw\le3^{3q}\sum_{k\le K-2}
3^{kq}\left|\Omega_{k+2}\setminus\Omega_{k+3}\right|_{
w},\end{equation*}
where
$\left|\Omega_{k+2}\setminus\Omega_{k+3}\right|_{w}:=\int_{\Omega_{k+2}\setminus\Omega_{k+3}}w$ and $\Omega_{K}\setminus\Omega_{K+1}=\Omega_K$, since $\Omega_{K+1}$ is empty.

Next, as in the proof of \cite[Theorem 1A]{Saw1}, we introduce rectangles. For this, we fix $ k $ such that $\Omega_{k+1}\not=\varnothing$, and choose points
$(x_j^k,y_j^k),$ $1\le j\le N=N_k,$ lying on the boundary $\partial\Omega_k$
in such a way to have $(x_{j}^k,y_{j-1}^k)$ belonging to $\partial\Omega_{k+1}$
for $2\le j\le N$ and $\Omega_{k+1}\subset\bigcup_{j=1}^NS_j^k,$
where $S_j^k$ is a rectangle of the form  $(x_j^k,\infty)\times(y_j^k,\infty)$. We also define rectangles $\widetilde{S}_{j}^k=(x_{j}^k,x_{j+1}^k)\times(y_{j}^k,y_{j-1}^k)$
for $1\le j\le N$ and $R_j^k=(0,x_{j+1}^k)\times(0,y_{j}^k)$, $\widetilde{R}_j^k=(x_j^k,x_{j+1}^k)\times(y_{j+1}^k,y_j^k)$ and $T_j^k=(x_{j+1}^k,\infty)\times(y_{j}^k,\infty)$ for $1\le j\le N-1$. Put $y_0^k=x_{N+1}^k=\infty$ (see Figure 1). 

Now we choose the sets $ E_j^k \subset T_j^k $ so that
$E_j^k\cap E_i^k=\varnothing$ for $j\not= i$ and
$\bigcup_jE_j^k=\left(\Omega_{k+2}\setminus\Omega_{k+3}\right)\cap\left(\bigcup_jT_j^k\right).$
Since
$\Omega_{k+2}\setminus\Omega_{k+3}\subset\Omega_{k+1}\subset\left(\bigcup_jT_j^k\right)\cup\left(\bigcup_j\widetilde{S}_j^k\right),$ then
\begin{eqnarray}\label{20}
3^{-3q}\int_{\mathbb{R}^2_+} (I_2f)^qw\le\sum_{k,j}
3^{kq}\big|E_j^k\big|_w+\sum_{k,j}
3^{kq}\big|\widetilde{S}_j^k\cap(\Omega_{k+2}-\Omega_{k+3})\big|_w=:I+II.\end{eqnarray}

To estimate $II$ we denote $D_j^k:=\widetilde{S}_j^k\setminus\Omega_{k+3}$ and turn to the reasoning of E. Sawyer on page 6 in \cite{Saw1}, from which it follows that \begin{equation*}\label{ub1}I_2(\chi_{D_j^k} f)(x,y)>3^k\quad\textrm{if}\quad (x,y)\in\widetilde{S}_j^k\cap(\Omega_{k+2}\setminus\Omega_{k+3}).\end{equation*}
Further, according to \cite[p. 6]{Saw1},
\begin{align}\label{PQ1}
\big|\widetilde{S}_j^k\cap(\Omega_{k+2}\setminus\Omega_{k+3})\big|_w&\le 3^{-k}\int_{\widetilde{S}_j^k\cap(\Omega_{k+2}\setminus\Omega_{k+3})} I_2(\chi_{D_j^k} f)(x,y)w(x,y)\,dxdy\nonumber\\&\le 3^{-k}\int_{D_j^k} \Bigl(\int_{x_j^k}^x\int_{y_j^k}^yf\Bigr)w(x,y)\,dxdy\nonumber\\&=3^{-k}\int_{D_j^k} f(s,t)\Bigl(\int_s^\infty\int_t^\infty w\chi_{D_j^k}\Bigr)\,dsdt\nonumber\\&\le 3^{-k}\biggl(\int_{D_j^k}f^pv\biggr)^{\frac{1}{p}} \biggl(\int_{D_j^k}\sigma(s,t)\Bigl(\int_s^\infty\int_t^\infty w\chi_{D_j^k}\Bigr)^{p'}\,dsdt\biggr)^{\frac{1}{p'}}.\end{align}
By applying Lemma \ref{lem2} to $ (a, b) \times (c, d) = \widetilde{S}_j^k $, we obtain for $ p <q $ that
\begin{equation}\label{PQ2}\mathbf{W}_{\widetilde{S}_j^k}=\int_{D_j^k}\sigma(s,t)\Bigl(\int_s^\infty\int_t^\infty w\chi_{D_j^k}\Bigr)^{p'}\,dsdt\le\alpha' A^{p'}\bigl|\widetilde{S}_j^k\bigr|_w^{\frac{p'}{q'}},\end{equation}
and in the case $q<p$ 
$$\mathbf{W}_{\widetilde{S}_j^k}\le\beta' \bigl|\widetilde{S}_j^k\bigr|_w^{\frac{p'}{q'}}\biggl(\int_{D_j^k}d_y\bigl[I_2\sigma(x,y)\bigr]^{\frac{r}{p'}}\,d_x\Bigl(-\bigl[I_2^\ast w(x,y)\bigr]^{\frac{r}{q}}\Bigr)\biggr)^{\frac{p'}{r}}.$$
For $ q <p $, from this and H\"older's inequality with $ q $ and $ q '$,
\begin{align*}(\beta')^{-\frac{1}{p'}}\cdot II\le\sum_{k,j}3^{k(q-1)}
\biggl(\int_{D_j^k}f^pv\biggr)^{\frac{1}{p}}
\biggl(\int_{D_j^k}d_y\bigl[I_2\sigma(x,y)\bigr]^{\frac{r}{p'}}\,d_x\Bigl(-\bigl[I_2^\ast w(x,y)\bigr]^{\frac{r}{q}}\Bigr)\biggr)^{\frac{1}{r}}\bigl|S_j^k\bigr|_w^{\frac{1}{q'}}\\\le\biggl(\sum_{k,j}3^{kq}\bigl|S_j^k\bigr|_w\biggr)^{\frac{1}{q'}}\Biggl[\sum_{k,j}
\biggl(\int_{D_j^k}f^pv\biggr)^{\frac{q}{p}}
\Biggl(\int_{D_j^k}d_y\bigl[I_2\sigma(x,y)\bigr]^{\frac{r}{p'}}\,d_x\Bigl(-\bigl[I_2^\ast w(x,y)\bigr]^{\frac{r}{q}}\Bigr)\Biggr)^{\frac{q}{r}}\Biggr]^{\frac{1}{q}}.
\end{align*}
On the strength of \cite[(2.6)]{Saw1}
$$\sum_{j=1}^{N_k}\chi_{S_j^k}\le 3^{-k}\chi_{\Omega_k}I_2f\qquad\textrm{for all} \ k.$$ Then \begin{multline*}
\sum_{k,j}3^{kq}\bigl|S_j^k\bigr|_w=\sum_{k}3^{kq}\sum_{j=1}^{N_k}\int_{\mathbb{R}_+^2}\chi_{S_j^k}w=\sum_{k}3^{kq}\int_{\mathbb{R}_+^2}\Bigl(\sum_{j=1}^{N_k}\chi_{S_j^k}\Bigr)w\le\sum_{k}3^{k(q-1)}\int_{\mathbb{R}_+^2}\chi_{\Omega_k}(I_2f)w\\=
\sum_{k}3^{k(q-1)}\sum_{m\ge k}\int_{\mathbb{R}_+^2}\chi_{\Omega_m\setminus\Omega_{m+1}}(I_2f)w=\sum_{m}3^{m(q-1)}\int_{\mathbb{R}_+^2}\chi_{\Omega_m\setminus\Omega_{m+1}}(I_2f)w\sum_{m\ge k}3^{(k-m)(q-1)}\\\le\frac{3^{q-1}}{3^{q-1}-1}\sum_{m}3^{m(q-1)}\int_{\mathbb{R}_+^2}\chi_{\Omega_m\setminus\Omega_{m+1}}(I_2f)w
\end{multline*} and, therefore,
$$
\sum_{k,j}3^{kq}\bigl|S_j^k\bigr|_w\le\frac{3^{q-1}}{3^{q-1}-1}\sum_{m}\int_{{\Omega_m\setminus\Omega_{m+1}}}\bigl(I_2f\bigr)^qw=\frac{3^{q-1}}{3^{q-1}-1}\int_{\mathbb{R}_+^2}\bigl(I_2f\bigr)^qw.
$$
Further, H\"older's inequality with $ p / q $, $ r / q $ and the estimate
$ \sum_{k, j} \chi_ {D_j^k} \le \sum_k \chi_{\Omega_k \setminus \Omega_{k + 3}} \le 3 $ entail
\begin{multline*}\sum_{k,j}
\biggl(\int_{D_j^k}f^pv\biggr)^{\frac{q}{p}}
\biggl(\int_{D_j^k}d_y\bigl[I_2\sigma(x,y)\bigr]^{\frac{r}{p'}}\,d_x\Bigl(-\bigl[I_2^\ast w(x,y)\bigr]^{\frac{r}{q}}\Bigr)\biggr)^{\frac{q}{r}}\\
\le\biggl(\sum_{k,j}
\int_{D_j^k}f^pv\biggr)^{\frac{q}{p}}\Biggl(\sum_{k,j}
\int_{D_j^k}d_y\bigl[I_2\sigma(x,y)\bigr]^{\frac{r}{p'}}\,d_x\Bigl(-\bigl[I_2^\ast w(x,y)\bigr]^{\frac{r}{q}}\Bigr)\biggr)^{\frac{q}{r}}\\\le 3\biggl(\int_{\mathbb{R}_+^2}f^pv\biggr)^{\frac{q}{p}}\Biggl(\int_{\mathbb{R}_+^2}d_y\bigl[I_2\sigma(x,y)\bigr]^{\frac{r}{p'}}\,d_x\Bigl(-\bigl[I_2^\ast w(x,y)\bigr]^{\frac{r}{q}}\Bigr)\Biggr)^{\frac{q}{r}}.\end{multline*} Thus, for $ q <p $, \begin{equation}\label{II}II\le3(\beta')^{\frac{1}{p'}}B\Bigl(\frac{3^{q-1}}{3^{q-1}-1}\Bigr)^{\frac{1}{q'}}\biggl(\int_{\mathbb{R}_+^2}f^pv\biggr)^{\frac{1}{p}}
\biggl(\int_{\mathbb{R}_+^2}\bigl(I_2 f\bigr)^qw\biggr)^{\frac{1}{q'}}.\end{equation} In the case $ p <q $ a similar estimate of the form \begin{equation}\label{IIA}II\le3^{\frac{1}{p}}(\alpha')^{\frac{1}{p'}} A\Bigl(\frac{3^{q-1}}{3^{q-1}-1}\Bigr)^{\frac{1}{q'}}\biggl(\int_{\mathbb{R}_+^2}f^pv\biggr)^{\frac{1}{p}}
\biggl(\int_{\mathbb{R}_+^2}\bigl(I_2 f\bigr)^qw\biggr)^{\frac{1}{q'}}\end{equation} follows from \eqref{PQ1}, \eqref{PQ2} and the reasoning on pages 6--7 in \cite{Saw1}.

To estimate $ I $ in \eqref{20}, in full accordance with the proof of \cite[Theorem 1A, pp. 8--9]{Saw1}, we put $ g \sigma: = f $ and write:
\begin{equation}\label{100} 3^{q}I=\sum_{k,j}
3^{(k+1)q}\big|{E}_j^k\big|_w=
\sum_{k,j} \bigl|{E}_j^k\bigr|_w\biggl(\int_{R_j^k}f\biggr)^q=\sum_{k,j}
\bigl|{E}_j^k\bigr|_w\bigl|R_j^k\bigr|_{\sigma}^q\biggl(\frac{1}{\bigl|R_j^k\bigr|_{\sigma}}
\int_{R_j^k}g\sigma\biggr)^q.
\end{equation}
For an integer $ l $, by $ \Gamma_l $ we denote the set of pairs $ (k, j) $ such that
$\bigl|{E}_j^k\bigr|_w>0$ and
\begin{eqnarray*}\label{110}
2^l<\frac{1}{\bigl|R_j^k\bigr|_{\sigma}}
\int_{R_j^k}g\sigma\leq 2^{l+1},\qquad(k,j)\in\Gamma_l
\end{eqnarray*}
and observe that $\Gamma_{l^\prime}\cap\Gamma_{l^{\prime\prime}}=\varnothing,$ $l^{\prime}\not=l^{\prime\prime}.$

For fixed $ l $ the family $ \{U_i^l \}_{i = 1}^{i (l)} $ consists of maximal rectangles from the collection $ \{R_j^k \}_{(k, j) \in \Gamma_l} $, that is, each $ R_j^k $ with $ (k, j) \in\Gamma_l $ is contained in some $ U_i^l $ (or coincides with it). In \cite[p. 8]{Saw1} it is shown that $\widetilde{U}_i^l$ are disjoint for fixed $ l $, where we denote  $\widetilde{U}_i^l=\widetilde{R}_i^l$ if  ${U}_i^l={R}_i^l$.

Let $\chi_i^l$ be the characteristic function of the union of the sets $E_j^k$ over all $(k,j)\in\Gamma_l$ such that $R_j^k\subset U_i^l$.
Further, following \cite[(2.13)]{Saw1}, we arrive to \begin{align}\label{cr4}
\sum_{(k,j)\in\Gamma_l} \bigl|{E}_j^k\bigr|_w\bigl|R_j^k\bigr|_{\sigma}^q
=&\sum_{i=1}^{i(l)}\sum_{(k,j)\colon R_j^k\subset U_i^l}\int_{{E}_j^k}w\bigl[I_2(\chi_{U_i^l}\sigma)(x_{j+1}^k,y_j^k)\bigr]^q\nonumber\\
\le& \sum_{i=1}^{i(l)}\int_{\mathbb{R}_+^2}\chi_i^lw\bigl[I_2(\chi_{U_i^l}\sigma)\bigr]^q.\end{align}

By analogy with \cite[(2.8)]{Saw1}, let us first show the validity of the estimate
\begin{equation}\label{cr3}
\int_{\mathbb{R}_+^2}\chi_i^lw\bigl[I_2(\chi_{U_i^l}\sigma)\bigr]^q\le \max\Bigl\{\beta,2q(p')^{q-1}\Bigl(\frac{q}{r}\Bigr)^{\frac{q}{r}}\Bigr\} \bigl(B_i^l\bigr)^q \bigl|U_i^l\bigr|_\sigma^{\frac{q}{p}}\end{equation} for $U_i^l=(0,a)\times(0,b)$ in the case $q<p$, where
\begin{equation*}
\bigl(B_i^l\bigr)^r=\int_{\mathbb{R}_+^2}\chi_i^l(x,y)d_y\bigl[I_2\sigma(x,y)\bigr]^{\frac{r}{p'}}\,d_x\Bigl(-\bigl[I_2^\ast w(x,y)\bigr]^{\frac{r}{q}}\Bigr).\end{equation*}
On $(0,a)\times(0,b)=U_i^l$, in view of Lemma \ref{lem1},
\begin{align*}\mathbf{V}_{U_i^l} =
\int_{U_i^l }\chi_i^lw\bigl(I_2\sigma\bigr)^{q}\le&\beta
\biggl(\int_{U_i^l }\chi_i^l(x,y)\,d_y\bigl[I_2\sigma(x,y)\bigr]^{\frac{r}{p'}}\,d_x\Bigl(-\bigl[I_2^\ast w(x,y)\bigr]^{\frac{r}{q}}\Bigr)\biggr)^{\frac{q}{r}}\bigl|U_i^l\bigr|_\sigma^{\frac{q}{p}}\nonumber\\\le&
\,\beta \bigl(B_i^l\bigr)^q \bigl|U_i^l\bigr|_\sigma^{\frac{q}{p}}.\end{align*} On the rectangle ${(a,\infty)\times(b,\infty)}$ we obtain the estimate:
\begin{multline*}
\int_{{(a,\infty)\times(b,\infty)}}\chi_i^lw\bigl|U_i^l\bigr|_\sigma^{q}=\Biggl(\int_{(a,\infty)\times(b,\infty)}\chi_i^l(x,y)\ d_xd_y\bigl[I^\ast_2 w\chi_i^l(x,y)\bigr]^{\frac{r}{q}}\Biggr)^{\frac{q}{r}}\bigl|U_i^l\bigr|_\sigma^{q}\\\le
\biggl(\int_{(a,\infty)\times(b,\infty)}\chi_i^l(x,y)\bigl[I_2 \sigma(x,y)\bigr]^{\frac{r}{p'}}\,d_x\,d_y\bigl[I^\ast_2 w\chi_i^l(x,y)\bigr]^{\frac{r}{q}}\biggr)^{\frac{q}{r}}\bigl|U_i^l\bigr|_\sigma^{\frac{q}{p}}\\\le
\biggl(\int_{\mathbb{R}_+^2}\chi_i^l(x,y)\bigl[I_2 \sigma(x,y)\bigr]^{\frac{r}{p'}}\,d_x\,d_y\bigl[I^\ast_2 w\chi_i^l(x,y)\bigr]^{\frac{r}{q}}\biggr)^{\frac{q}{r}}\bigl|U_i^l\bigr|_\sigma^{\frac{q}{p}},\end{multline*} whence by integration by parts
\begin{equation*}
\int_{\mathbb{R}_+^2}\chi_i^l(x,y)\bigl[I_2 \sigma(x,y)\bigr]^{\frac{r}{p'}}\,d_x\,d_y\bigl[I^\ast_2 w\chi_i^l(x,y)\bigr]^{\frac{r}{q}}=\bigl(B_i^l\bigr)^r.\end{equation*}
In the first of the two mixed cases --- ${(0,a)\times(b,\infty)}$ and ${(a,\infty)\times(0,b)}$ --- we obtain, using the criteria for the fulfillment of the one--dimensional weighted Hardy inequality for $f^p(x)=\int_0^b\sigma(x,y)\,dy$ (see \cite[\S\,1.3.2]{Maz}): \begin{align}&
\int_{{(0,a)\times(b,\infty)}}\chi_i^l(x,y)w(x,y)\biggl(\int_0^x\int_0^b\sigma\biggr)^{q}\,dxdy\nonumber\\&=\int_0^a\biggl(\int_b^\infty\chi_i^l(x,y)w(x,y)\,dy\biggr)\biggl(\int_0^x\biggl(\int_0^b\sigma(s,t)\,dt\biggr)\biggr)^{q}\,dx\nonumber\\&\le q(p')^{q-1}\biggl(\int_0^a\biggl(\int_s^\infty\int_b^\infty\chi_i^lw\biggr)^{\frac{r}{p}}\biggl(\int_0^s\int_0^b\sigma\biggr)^{\frac{r}{p'}}\biggl(\int_b^\infty\chi_i^l(s,t)w(s,t)\,dt\biggr)ds\biggr)^{\frac{q}{r}}\bigl|U_i^l\bigr|_\sigma^{\frac{q}{p}}\nonumber\\&=q(p')^{q-1}\Bigl(\frac{q}{r}\Bigr)^{\frac{q}{r}}\biggl(\int_0^a\biggl(\int_0^s\int_0^b\sigma\biggr)^{\frac{r}{p'}}d_s\biggl[-\biggl(\int_s^\infty\int_b^\infty\chi_i^lw\biggr)^{\frac{r}{q}}\biggr]\biggr)^{\frac{q}{r}}\bigl|U_i^l\bigr|_\sigma^{\frac{q}{p}}\nonumber\\&=q(p')^{q-1}\Bigl(\frac{q}{r}\Bigr)^{\frac{q}{r}}\biggl(\int_0^a\biggl(\int_0^s\int_0^b\sigma\biggr)^{\frac{r}{p'}}d_s\biggl[-\int_b^\infty d_t\biggl[-\biggl(\int_s^\infty\int_t^\infty\chi_i^lw\biggr)^{\frac{r}{q}}\biggr]\biggr]\biggr)^{\frac{q}{r}}\bigl|U_i^l\bigr|_\sigma^{\frac{q}{p}}\nonumber\\&=q(p')^{q-1}\Bigl(\frac{q}{r}\Bigr)^{\frac{q}{r}}\biggl(\int_0^a\int_b^\infty \chi_i^l(s,t)\biggl(\int_0^s\int_0^b\sigma\biggr)^{\frac{r}{p'}}\,d_sd_t\biggl(\int_s^\infty\int_t^\infty\chi_i^lw\biggr)^{\frac{r}{q}}\biggr)^{\frac{q}{r}}\bigl|U_i^l\bigr|_\sigma^{\frac{q}{p}}\nonumber\end{align} \begin{align}&\leq(p')^{q-1}\Bigl(\frac{q}{r}\Bigr)^{\frac{q}{r}}\biggl(\int_0^a\int_b^\infty\chi_i^l(s,t)\bigl[I_2\sigma(s,t)\bigr]^{\frac{r}{p'}}\,d_sd_t\biggl(\int_s^\infty\int_t^\infty\chi_i^lw\biggr)^{\frac{r}{q}}\biggr)^{\frac{q}{r}}\bigl|U_i^l\bigr|_\sigma^{\frac{q}{p}}\nonumber\\&\le q(p')^{q-1}\Bigl(\frac{q}{r}\Bigr)^{\frac{q}{r}}\biggl(\int_{\mathbb{R}_+^2}\chi_i^l(s,t)\bigl[I_2\sigma(s,t)\bigr]^{\frac{r}{p'}}\,d_sd_t\biggl(\int_s^\infty\int_t^\infty\chi_i^lw\biggr)^{\frac{r}{q}}\biggr)^{\frac{q}{r}}\bigl|U_i^l\bigr|_\sigma^{\frac{q}{p}}\nonumber\\&\le q(p')^{q-1}\Bigl(\frac{q}{r}\Bigr)^{\frac{q}{r}}
\biggl(\int_{\mathbb{R}_+^2}\chi_i^l(s,t)\,d_t\bigl[I_2\sigma(s,t)\bigr]^{\frac{r}{p'}}\,d_s\Bigl(-\bigl[I_2^\ast w(s,t)\bigr]^{\frac{r}{q}}\Bigr)\biggr)^{\frac{q}{r}}\bigl|U_i^l\bigr|_\sigma^{\frac{q}{p}}.\end{align}
The second mixed case is estimated in a similar way. So, \eqref{cr3} is proven.
Continuing \eqref{cr4}, we obtain, using \cite [(2.11)] {Saw1} and H\"older's inequality with $ r / q $, $ p / q $:
\begin{align*}
\sum_{(k,j)\in\Gamma_l} \bigl|{E}_j^k\bigr|_w\bigl|R_j^k\bigr|_{\sigma}^q
\le&\max\Bigl\{\beta,2q(p')^{q-1}\Bigl(\frac{q}{r}\Bigr)^{\frac{q}{r}}\Bigr\}\sum_i\bigl(B_i^l\bigr)^{q}\bigl|U_i^l\bigr|_\sigma^{\frac{q}{p}}\\
\le&\max\Bigl\{\beta,2q(p')^{q-1}\Bigl(\frac{q}{r}\Bigr)^{\frac{q}{r}}\Bigr\}\sum_i\bigl(B_i^l\bigr)^{q}\Bigl(2^{-l}\int_{\widetilde{U}_i^l\cap\{g>2^{l-3}\}}g\sigma\Bigr)^{\frac{q}{p}} \\
\le&\max\Bigl\{\beta,2q(p')^{q-1}\Bigl(\frac{q}{r}\Bigr)^{\frac{q}{r}}\Bigr\}\biggl(\sum_i\bigl(B_i^l\bigr)^{r}\biggr)^{\frac{q}{r}}\biggl(\sum_i 2^{-l}\int_{\widetilde{U}_i^l\cap\{g>2^{l-3}\}}g\sigma\biggr)^{\frac{q}{p}}\\\le& \max\Bigl\{\beta,2q(p')^{q-1}\Bigl(\frac{q}{r}\Bigr)^{\frac{q}{r}}\Bigr\}2^{-l{q}/{p}}\bigl(B_l\bigr)^{q}\biggl(\int_{\{g>2^{l-3}\}}g\sigma\biggr)^{\frac{q}{p}}.
\end{align*}
The last estimate is valid with
\begin{equation*}
\bigl(B_l\bigr)^r:=\int_{\mathbb{R}_+^2}\chi_{\{\cup_{(k,j)\in\Gamma_l} E_j^k\}}(x,y)d_y\bigl[I_2\sigma(x,y)\bigr]^{\frac{r}{p'}}\,d_x\Bigl(-\bigl[I_2^\ast w(x,y)\bigr]^{\frac{r}{q}}\Bigr)
\end{equation*}
due to the fact that for fixed $ l $ the rectangles $ \widetilde{U}_i^l $ do not intersect (see \cite[p. 8]{Saw1}). Combining it with \eqref{100}, we obtain, taking into account the relation
 $$\sum_l2^{l(p-1)}\chi_{\{g>2^{l-3}\}}\le\frac{2^{p-1}}{2^{p-1}-1}\, g^{p-1}\qquad\textrm{for}\quad p>1,
 $$
 H\"older's inequality with $ r / q $ and $ p / q $ and the fact that all $ E_j^k $ are disjoint: \begin{align}\label{I}
I\le&\Bigl(\frac{2}{3}\Bigr)^{q} \sum_l2^{lq}\sum_{(k,j)\in\Gamma_l} \bigl|{E}_j^k\bigr|_w\bigl|R_j^k\bigr|_{\sigma}^q\nonumber\\\le&\Bigl(\frac{2}{3}\Bigr)^{q}\max\Bigl\{\beta,2q(p')^{q-1}\Bigl(\frac{q}{r}\Bigr)^{\frac{q}{r}}\Bigr\}\sum_l 2^{lq}\bigl(B_l\bigr)^{q}\biggl(2^{-l}\int_{\{g>2^{l-3}\}}g\sigma\biggr)^{\frac{q}{p}}\nonumber\\
\le&\Bigl(\frac{2}{3}\Bigr)^{q}\max\Bigl\{\beta,2q(p')^{q-1}\Bigl(\frac{q}{r}\Bigr)^{\frac{q}{r}}\Bigr\}\Bigl(\sum_l \bigl(B_l\bigr)^{r}\Bigr)^{\frac{q}{r}}\biggl(\sum_l2^{l(p-1)}\int_{\{g>2^{l-3}\}}g\sigma\biggr)^{\frac{q}{p}}\nonumber\\
\le& \Bigl(\frac{2}{3}\Bigr)^{q}\max\Bigl\{\beta,2q(p')^{q-1}\Bigl(\frac{q}{r}\Bigr)^{\frac{q}{r}}\Bigr\}
\Bigl(\frac{2^{p-1}}{2^{p-1}-1}\Bigr)^{\frac{q}{p}}B^q\biggl(\int_{\mathbb{R}_+^2}f^pv\biggr)^{\frac{q}{p}}.\end{align} Combining \eqref{I} with \eqref{II} we arrive at the required upper bound for $ q <p $.

For $ p <q $, the term $ I $ is estimated identically to the case $ p \le q $ in \cite[p. 9]{Saw1}, i.e.
\begin{equation}\label{IA}
I\le\Bigl(\frac{2}{3}\Bigr)^{q}\max\Bigl\{\alpha,2q(q')^{\frac{q}{p'}}\Bigr\}
\Bigl(\frac{2^{p-1}}{2^{p-1}-1}\Bigr)^{\frac{q}{p}} A^q\biggl(\int_{\mathbb{R}_+^2}f^pv\biggr)^{\frac{q}{p}},\end{equation}
relying on an analog of the inequality \eqref{cr3} of the form
\begin{equation*}\label{cr3'''}
\int_{\mathbb{R}_+^2}\chi_i^lw\bigl[I_2(\chi_{U_i^l}\sigma)\bigr]^q\le\max\Bigl\{\alpha,2q(q')^{\frac{q}{p'}}\Bigr\} A^q \bigl|U_i^l\bigr|_\sigma^{\frac{q}{p}}\quad\textrm{for}\ U_i^l=(0,a)\times(0,b).\end{equation*} Note that in this case, unlike \cite[(2.8)]{Saw1}, to perform the estimate on the rectangle $(0, a)\times (0, b)=U_i^l$ one should apply the statement of Lemma \ref{lem1}, from which it follows that
\begin{equation*}\mathbf{V}_{U_i^l} \le \alpha
\bigl|U_i^l\bigr|_\sigma^{\frac{q}{p}} A^q.\end{equation*}

The final upper estimate $$\int_{\mathbb{R}_+^2}(I_2f)^qw\le C\biggl( \int_{\mathbb{R}_+^2}f^pv\biggr)^{\frac{1}{p}}\biggl(\int_{\mathbb{R}_+^2}(I_2f)^qw\biggr)^{\frac{1}{q'}}+C^q\biggl( \int_{\mathbb{R}_+^2}f^pv\biggr)^{\frac{q}{p}}$$ follows from \eqref{20} combined with \eqref{IIA} and \eqref{IA} for $ p <q $ (or \eqref{II} and \eqref{I} if $ q <p $) with $ C = A \cdot \mathbb{C}_{\alpha, \alpha'} $ in case $ p <q $ and $ C = B \cdot \mathbf{C}_{\beta,\beta'} $ for $ q <p $.

(\textit{Necessity}) The validity of $ A \le C_2 $ follows by substituting $f=\chi_{(0,s)\times(0,t)}$ into the initial inequality \eqref{aux}. To establish $ B \lesssim C_2 $ in the case $ q <p $, we apply the test function
$$
f(s,y)=\sigma(s,y)\biggl[\int_s^\infty \bigl[I_2\sigma(x,y)\bigr]^{\frac{r}{q'}}\bigl[I_2^\ast w(x,y)\bigr]^{\frac{r}{p}}
\Bigl(\int_y^\infty w(x,t)\,dt\Bigr)dx\biggr]^{\frac{1}{p}}=:
{\sigma(s,y)J(s,y)}
$$
into \eqref{aux}. Then
\begin{align}\label{cr2}
\int_{\mathbb{R}_+^2}f^{p}v=&\int_{\mathbb{R}^2_+}\sigma(s,y){\bigl[J(s,y)\bigr]^p}dsdy\nonumber\\=&
\int_{\mathbb{R}^2_+} \bigl[I_2\sigma(x,y)\bigr]^{\frac{r}{q'}}\bigl[I_2^\ast w(x,y)\bigr]^{\frac{r}{p}}
\Bigl(\int_y^\infty w(x,t)\,dt\Bigr)\Bigl(\int_0^x\sigma(s,y)ds\Bigr)dxdy \nonumber\\=&\frac{p'q}{r^2}
\int_{\mathbb{R}^2_+} d_y\bigl[I_2\sigma(x,y)\bigr]^{\frac{r}{p'}}\,d_x\Bigl[-\bigl[I_2^\ast w(x,y)\bigr]^{\frac{r}{q}}\Bigr]=\frac{p'q}{r^2}\,B^{r}.\end{align}

To estimate the left--hand side of the inequality \eqref{aux}, we write
\begin{align}\label{cr6'}
{\bigl[J(s,y)\bigr]^p}=&\frac{q}{r}
\bigl[I_2\sigma(s,y)\bigr]^{\frac{r}{q'}}\bigl[I_2^\ast w(s,y)\bigr]^{\frac{r}{q}}\nonumber\\
&+\frac{q}{q'}\int_s^\infty \bigl[I_2\sigma(x,y)\bigr]^{\frac{r}{q'}-1}\bigl[I_2^\ast w(x,y)\bigr]^{\frac{r}{q}}\Bigl(\int_0^y\sigma(x,t)\,dt\Bigr)dx\nonumber\\
=&:\frac{q}{r}
\bigl[J_1(s,y)\bigr]^p+\frac{q}{q'}\bigl[J_2(s,y)\bigr]^p.
\end{align}
Then, for our chosen $ f $,
\begin{align*}
F(u,z):=&\int_0^u\int_0^zf=\int_0^u\int_0^z
\sigma(s,y){J(s,y)}\,dyds
\\
\geq& 2^{-\frac{1}{p'}}\biggl(
\Bigl(\frac{q}{r}\Bigr)^{\frac{1}{p}}\int_0^u\int_0^z
\sigma(s,y){
J_1(s,y)}\,dyds+\Bigl(\frac{q}{q'}\Bigr)^{\frac{1}{p}}
\int_0^u\int_0^z
\sigma(s,y){
J_2(s,y)}\,dyds\biggr)\\=&:
2^{-\frac{1}{p'}}\bigl(F_1+F_2 \bigr).
\end{align*}
To estimate $ F_2 $, we observe that
\begin{align*}
\Bigl(\frac{q'}{q}\Bigr)^{\frac{1}{p}}F_2=&
\int_0^u\int_0^z
\sigma(s,y){
J_2(s,y)}\,dyds\\
\ge& \bigl[I_2^\ast w(u,z)\bigr]^{\frac{r}{qp}}\int_0^u\int_0^z
\sigma(s,y)\biggl[\int_s^u \bigl[I_2\sigma(x,y)\bigr]^{\frac{r}{q'}-1}
\Bigl(\int_0^y\sigma(x,t)\,dt\Bigr)dx\biggr]^{\frac{1}{p}}dyds.\end{align*} Since
\begin{equation}\label{cr7}\int_s^u \bigl[I_2\sigma(x,y)\bigr]^{\frac{r}{q'}-1}
\Bigl(\int_0^y\sigma(x,t)\,dt\Bigr)dx\le\frac{q'}{r}\bigl[I_2\sigma(u,y)\bigr]^{\frac{r}{q'}},\end{equation} then
\begin{multline*}
\int_0^u\int_0^z
\sigma(s,y)\biggl[\int_s^u \bigl[I_2\sigma(x,y)\bigr]^{\frac{r}{q'}-1}
\Bigl(\int_0^y\sigma(x,t)\,dt\Bigr)dx\biggr]^{1-\frac{1}{p'}}dyds\\\ge \Bigl(\frac{q'}{r}\Bigr)^{-\frac{1}{p'}}
\int_0^u\int_0^z
\sigma(s,y)\bigl[I_2\sigma(u,y)\bigr]^{-\frac{r}{q'p'}}\biggl[\int_s^u \bigl[I_2\sigma(x,y)\bigr]^{\frac{r}{q'}-1}
\Bigl(\int_0^y\sigma(x,t)\,dt\Bigr)dx\biggr]dyds\\\ge
\Bigl(\frac{q'}{r}\Bigr)^{-\frac{1}{p'}}\bigl[I_2\sigma(u,z)\bigr]^{-\frac{r}{q'p'}}
\int_0^u\int_0^z
\sigma(s,y)\biggl[\int_s^u \bigl[I_2\sigma(x,y)\bigr]^{\frac{r}{q'}-1}
\Bigl(\int_0^y\sigma(x,t)\,dt\Bigr)dx\biggr]dyds\\=
\Bigl(\frac{q'}{r}\Bigr)^{-\frac{1}{p'}}\bigl[I_2\sigma(u,z)\bigr]^{-\frac{r}{q'p'}}
\int_0^u\int_0^z\bigl[I_2\sigma(x,y)\bigr]^{\frac{r}{q'}-1}
\Bigl(\int_0^x\sigma(s,y)\,ds\Bigr)\Bigl(\int_0^y\sigma(x,t)\,dt\Bigr)dydx
\end{multline*}
and, therefore,
\begin{align*}
F_2\ge& \Bigl(\frac{q}{q'}\Bigr)^{\frac{1}{p}}\Bigl(\frac{r}{q'}\Bigr)^{\frac{1}{p'}}\bigl[I_2\sigma(u,z)\bigr]^{-\frac{r}{q'p'}}\bigl[I_2^\ast w(u,z)\bigr]^{\frac{r}{qp}}\\
&\times\int_0^u\int_0^z\bigl[I_2\sigma(x,y)\bigr]^{\frac{r}{q'}-1}
\Bigl(\int_0^x\sigma(s,y)\,ds\Bigr)\Bigl(\int_0^y\sigma(x,t)\,dt\Bigr)dxdy\\
=&:\Bigl(\frac{q}{r}\Bigr)^{\frac{1}{p}}\frac{r}{q'}\bigl[I_2\sigma(u,z)\bigr]^{-\frac{r}{q'p'}}\bigl[I_2^\ast w(u,z)\bigr]^{\frac{r}{qp}}{\mathbf {J}_2(u,z)}.\end{align*}
For $F_1$ we obtain:
\begin{align*}
F_1=&\Bigl(\frac{q}{r}\Bigr)^{\frac{1}{p}}\int_0^u\int_0^z
\sigma(s,y)\bigl[I_2\sigma(s,y)\bigr]^{\frac{r}{q'p}}\bigl[I_2^\ast w(s,y)\bigr]^{\frac{r}{qp}}dyds\\
\ge&\Bigl(\frac{q}{r}\Bigr)^{\frac{1}{p}}\bigl[I_2\sigma(u,z)\bigr]^{-\frac{r}{q'p'}}\bigl[I_2^\ast w(u,z)\bigr]^{\frac{r}{qp}}\int_0^u\int_0^z
\sigma(s,y)\bigl[I_2\sigma(s,y)\bigr]^{\frac{r}{q'}}dyds\\
=&:\Bigl(\frac{q}{r}\Bigr)^{\frac{1}{p}}\bigl[I_2\sigma(u,z)\bigr]^{-\frac{r}{q'p'}}\bigl[I_2^\ast w(u,z)\bigr]^{\frac{r}{qp}}{
\mathbf {J}_1(u,z)}.\end{align*}
{
It holds that
$$
F(u,z)\geq 2^{-\frac{1}{p'}}\Bigl(\frac{q}{r}\Bigr)^{\frac{1}{p}}\bigl[I_2\sigma(u,z)\bigr]^{-\frac{r}{q'p'}}\bigl[I_2^\ast w(u,z)\bigr]^{\frac{r}{qp}}\bigl({
\mathbf {J}_1(u,z)}+\frac{r}{q'}{\mathbf {J}_2(u,z)}
\bigr).
$$
Integrating by parts we find:
\begin{align*}
{\mathbf {J}_2(u,z)}=&\frac{q'}{r}\int_0^udx\int_0^z\Bigl(\int_0^y\sigma(x,t)\,dt\Bigr)d_y
\bigl[I_2\sigma(x,y)\bigr]^{\frac{r}{q'}}\\
=&\frac{q'}{r}\int_0^u\Bigl(\int_0^z\sigma(x,t)\,dt\Bigr)
\bigl[I_2\sigma(x,z)\bigr]^{\frac{r}{q'}}dx-\frac{q'}{r}{
\mathbf {J}_1(u,z)}\\
=&\frac{q'p'}{r^2}\bigl[I_2\sigma(u,z)\bigr]^{\frac{r}{p'}}-\frac{q'}{r}{
\mathbf {J}_1(u,z)}.\end{align*}
Hence,
\begin{equation}\label{cr8}
F(u,z)\ge 2^{-\frac{1}{p'}} \Bigl(\frac{q}{r}\Bigr)^{\frac{1}{p}}\frac{p'}{r}\bigl[I_2\sigma(u,z)\bigr]^{\frac{r}{qp'}}\bigl[I_2^\ast w(u,z)\bigr]^{\frac{r}{qp}}.
\end{equation}
}
We write making use of \eqref{cr6'}:
\begin{align}\label{cr9}
\int_{\mathbb{R}^2_+}(I_2f)^{q}w=& \int_{\mathbb{R}^2_+}f(x,y)\biggl(\int_x^\infty\int_y^\infty w(u,z)\bigl[F(u,z)\bigr]^{q-1}dzdu\biggr)\,dxdy\nonumber\\
 \ge& 2^{-\frac{1}{p'}}
\int_{\mathbb{R}^2_+}\sigma(x,y)\biggl(\int_x^\infty\int_y^\infty wF^{q-1}\biggr)\biggl\{\Bigl(\frac{q}{r}\Bigr)^{\frac{1}{p}}
\bigl[I_2\sigma(x,y)\bigr]^{\frac{r}{q'p}}\bigl[I_2^\ast w(x,y)\bigr]^{\frac{r}{qp}}\nonumber\\&+\Bigl(\frac{q}{q'}\Bigr)^{\frac{1}{p}}\biggl[\int_x^\infty \bigl[I_2\sigma(s,y)\bigr]^{\frac{r}{q'}-1}\bigl[I_2^\ast w(s,y)\bigr]^{\frac{r}{q}}
\Bigl(\int_0^y\sigma(s,t)\,dt\Bigr)ds\biggr]^{\frac{1}{p}}\biggr\}dxdy \nonumber\\=&: 2^{-\frac{1}{p'}}\bigl(G_1+G_2\bigr).
\end{align}
$ G_1 $ is evaluated with \eqref{cr8} as follows:
\begin{align}\label{cr13}
G_1=&\Bigl(\frac{q}{r}\Bigr)^{\frac{1}{p}}
\int_{\mathbb{R}^2_+}\sigma(x,y)
\bigl[I_2\sigma(x,y)\bigr]^{\frac{r}{q'p}}\bigl[I_2^\ast w(x,y)\bigr]^{\frac{r}{qp}}\biggl(\int_x^\infty\int_y^\infty w\biggr)\bigl[F(x,y)\bigr]^{q-1}\,dxdy\nonumber\\\ge& 2^{-\frac{q-1}{p'}}\Bigl(\frac{q}{r}\Bigr)^{\frac{q}{p}}\Bigl(\frac{p'}{r}\Bigr)^{{q-1}}
\int_{\mathbb{R}^2_+}\sigma(x,y)
\bigl[I_2\sigma(x,y)\bigr]^{\frac{r}{q'}}\bigl[I_2^\ast w(x,y)\bigr]^{\frac{r}{q}}\,dxdy.\end{align}
It is true for $ G_2 $:
\begin{align*}\Bigl(\frac{q'}{q}\Bigr)^{\frac{1}{p}}G_2=&
\int_{\mathbb{R}^2_+}\sigma(x,y)\biggl[\int_x^\infty \bigl[I_2\sigma(s,y)\bigr]^{\frac{r}{q'}-1}\bigl[I_2^\ast w(s,y)\bigr]^{\frac{r}{q}}
\Bigl(\int_0^y\sigma(s,t)\,dt\Bigr)ds\biggr]^{\frac{1}{p}}\nonumber\\ &\times\biggl(\int_x^\infty\int_y^\infty w(u,z)\bigl[F(u,z)\bigr]^{q-1}dzdu\biggr)\,dxdy\nonumber\\=&
\int_{\mathbb{R}^2_+}\int_0^u\sigma(x,y)\biggl[\int_x^\infty \bigl[I_2\sigma(s,y)\bigr]^{\frac{r}{q'}-1}\bigl[I_2^\ast w(s,y)\bigr]^{\frac{r}{q}}
\Bigl(\int_0^y\sigma(s,t)\,dt\Bigr)ds\biggr]^{\frac{1}{p}}dx\nonumber\\&\times\biggl(\int_y^\infty w(u,z)\bigl[F(u,z)\bigr]^{q-1}dz\biggr)\,dudy \nonumber\\\ge&
\int_{\mathbb{R}^2_+}\int_0^u\sigma(x,y)\biggl[\int_x^u \bigl[I_2\sigma(s,y)\bigr]^{\frac{r}{q'}-1}\bigl[I_2^\ast w(s,y)\bigr]^{\frac{r}{q}}
\Bigl(\int_0^y\sigma(s,t)\,dt\Bigr)ds\biggr]^{\frac{1}{p}}dx \nonumber\\&\times\biggl(\int_y^\infty w(u,z)\bigl[F(u,z)\bigr]^{q-1}dz\biggr)\,dudy \nonumber\\\ge&
\int_{\mathbb{R}^2_+}\bigl[I_2^\ast w(u,y)\bigr]^{\frac{r}{pq}}\int_0^u\sigma(x,y)\biggl[\int_x^u \bigl[I_2\sigma(s,y)\bigr]^{\frac{r}{q'}-1}
\Bigl(\int_0^y\sigma(s,t)\,dt\Bigr)ds\biggr]^{1-\frac{1}{p'}}dx\nonumber\\&\times\biggl(\int_y^\infty w(u,z)\bigl[F(u,z)\bigr]^{q-1}dz\biggr)\,dudy\\\overset{\eqref{cr7}}\ge&
\Bigl(\frac{r}{q'}\Bigr)^{\frac{1}{p'}}\int_{\mathbb{R}^2_+}\bigl[I_2\sigma(u,y)\bigr]^{-\frac{r}{q'p'}}
\bigl[I_2^\ast w(u,y)\bigr]^{\frac{r}{qp}}\\&\times \int_0^u\sigma(x,y)\biggl[\int_x^u \bigl[I_2\sigma(s,y)\bigr]^{\frac{r}{q'}-1}
\Bigl(\int_0^y\sigma(s,t)\,dt\Bigr)ds\biggr]dx\end{align*} \begin{align*}&\times\biggl(\int_y^\infty w(u,z)\bigl[F(u,z)\bigr]^{q-1}dz\biggr)\,dudy\\
\ge&\Bigl(\frac{r}{q'}\Bigr)^{\frac{1}{p'}}\int_{\mathbb{R}^2_+}
\bigl[I_2\sigma(u,y)\bigr]^{-\frac{r}{q'p'}}
\bigl[I_2^\ast w(u,y)\bigr]^{\frac{r}{qp}}\biggl(\int_y^\infty w(u,z)\,dz\biggr)\bigl[F(u,y)\bigr]^{q-1}\\&\times\biggl[\int_0^u \bigl[I_2\sigma(s,y)\bigr]^{\frac{r}{q'}-1}\Bigl(\int_0^s\sigma(x,y)\,dx\Bigr)
\Bigl(\int_0^y\sigma(s,t)\,dt\Bigr)ds\biggr]\,dudy.\end{align*}
Integrating by parts we find
\begin{multline*}
\int_0^u \bigl[I_2\sigma(s,y)\bigr]^{\frac{r}{q'}-1}\Bigl(\int_0^s\sigma(x,y)\,dx\Bigr)
\Bigl(\int_0^y\sigma(s,t)\,dt\Bigr)ds\\
=\frac{q'}{r}\Bigl(\int_0^u\sigma(x,y)\,dx\Bigr)
\bigl[I_2\sigma(u,y)\bigr]^{\frac{r}{q'}}dx-\frac{q'}{r}\int_0^u\bigl[I_2\sigma(s,y)\bigr]^{\frac{r}{q'}}\sigma(s,y)\,ds.
\end{multline*}
Hence, continuing the reasoning, we obtain for $ G_2 $ using \eqref{cr8}:
\begin{align}\label{cr13'}
\Bigl(\frac{q'}{q}\Bigr)^{\frac{1}{p}}G_2\ge& 2^{-\frac{q-1}{p'}}\Bigl(\frac{q'}{r}\Bigr)^{\frac{1}{p}}\Bigl(\frac{q}{r}\Bigr)^{\frac{q-1}{p}}\Bigl(\frac{p'}{r}\Bigr)^{q-1}\int_{\mathbb{R}^2_+}
\bigl[I_2^\ast w(u,y)\bigr]^{\frac{r}{p}}\biggl(\int_y^\infty w(u,z)\,dz\biggr)\nonumber\\&\times\biggl[\bigl[I_2\sigma(u,y)\bigr]^{\frac{r}{q'}}\int_0^u \sigma(x,y)\,dx-\int_0^u\bigl[I_2\sigma(s,y)\bigr]^{\frac{r}{q'}}\sigma(s,y)ds\biggr]\,dudy.
\end{align}
Since
\begin{multline*}
\int_{\mathbb{R}^2_+}
\bigl[I_2^\ast w(u,y)\bigr]^{\frac{r}{p}}\biggl(\int_y^\infty w(u,z)\,dz\biggr)\biggl[\int_0^u\bigl[I_2\sigma(s,y)\bigr]^{\frac{r}{q'}}\sigma(s,y)\,ds\biggr]\,dudy\\=\frac{q}{r}
\int_{\mathbb{R}^2_+}
\bigl[I_2^\ast w(u,y)\bigr]^{\frac{r}{q}}\bigl[I_2\sigma(u,y)\bigr]^{\frac{r}{q'}}\sigma(u,y)\,dudy
\end{multline*}
then from \eqref{cr9} we obtain, applying \eqref{cr13} and \eqref{cr13'},
\begin{multline*}
2^{\frac{q}{p'}}\int_{\mathbb{R}^2_+}(I_2f)^{q}w\ge 
\Bigl(\frac{q}{r}\Bigr)^{\frac{q}{p}}\Bigl(\frac{p'}{r}\Bigr)^{{q-1}}
\int_{\mathbb{R}^2_+}\sigma(x,y)
\bigl[I_2\sigma(x,y)\bigr]^{\frac{r}{q'}}\bigl[I_2^\ast w(x,y)\bigr]^{\frac{r}{q}}\,dxdy\\
+\Bigl(\frac{q}{r}\Bigr)^{\frac{q}{p}}\Bigl(\frac{p'}{r}\Bigr)^{\frac{q}{q'}}
\int_{\mathbb{R}^2_+}
\bigl[I_2^\ast w(u,y)\bigr]^{\frac{r}{p}}\biggl(\int_y^\infty w(u,z)\,dz\biggr)\bigl[I_2\sigma(u,y)\bigr]^{\frac{r}{q'}}\biggl(\int_0^u \sigma(x,y)\,dx\biggr)dudy\\
-\Bigl(\frac{q}{r}\Bigr)^{\frac{q}{p}}\Bigl(\frac{p'}{r}\Bigr)^{{q-1}}\frac{q}{r}\int_{\mathbb{R}^2_+}\sigma(x,y)
\bigl[I_2\sigma(x,y)\bigr]^{\frac{r}{q'}}\bigl[I_2^\ast w(x,y)\bigr]^{\frac{r}{q}}\,dxdy\\
=\Bigl(\frac{q}{r}\Bigr)^{\frac{q}{p}}\Bigl(\frac{p'}{r}\Bigr)^{{q-1}}\frac{q}{p}\int_{\mathbb{R}^2_+}\sigma(x,y)
\bigl[I_2\sigma(x,y)\bigr]^{\frac{r}{q'}}\bigl[I_2^\ast w(x,y)\bigr]^{\frac{r}{q}}\,dxdy\\
+\Bigl(\frac{q}{r}\Bigr)^{\frac{q}{p}+1}\Bigl(\frac{p'}{r}\Bigr)^{{q}}
\int_{\mathbb{R}^2_+}
d_u\Bigl(-\bigl[I_2^\ast w(u,y)\bigr]^{\frac{r}{q}}\Bigr)d_y\bigl[I_2\sigma(u,y)\bigr]^{\frac{r}{p'}}\ge 
\Bigl(\frac{q}{r}\Bigr)^{\frac{q}{p}+1}\Bigl(\frac{p'}{r}\Bigr)^{{q}}B^r.
\end{multline*}
In view of \eqref{cr2}, the required lower bound for $ C_2 $ in the case $ q <p $ is proven.
\end{proof}

Recall that in the case $ p \le q $ the best constant $ C_2 $ of the two--dimensional inequality \eqref{aux} is equivalent to $ \sum_{i = 1}^3A_i $ (see Theorem \ref{ES}). However, by virtue of the statements of Lemmas \ref{lem1} and \ref{lem2}, for $ p <q $ the following inequalities take place:
\begin{equation}\label{AAA}A_1\le C_2\le \mathbb{C}_{1,1}\bigl[ A_1+A_2+A_3\bigr]\le \mathbb{C}_{1,1}\bigl[1+\alpha(p,q)^{\frac{1}{q}}+\alpha(q',p')^{\frac{1}{p'}}\bigr]A_1.
\end{equation} Moreover,
$$\lim_{p\uparrow q}\bigl[\alpha(p,q)+\alpha(q',p')\bigr]=\infty.
$$
Thus, the last estimate in \eqref{AAA} and the upper bound in the main theorem have blow-up for $p\uparrow q $.

Estimates similar to \eqref {AAA} hold also in the case $ q <p $ if conditions $ r / p \ge 1 $ and $ r / q'\ge 1 $ are simultaneously satisfied, namely,
\begin{equation}\label{BBB}\Bigl(\frac{q}{r}\Bigr)^{\frac{1}{q}}\Bigl(\frac{p'}{2r}\Bigr)^{\frac{1}{p'}}B_1\le C_2\le\mathbf{C}_{1,1}\bigl[ B_1+B_2+B_3\bigr]\le \mathbf{C}_{1,1}\bigl[1+\boldsymbol{\beta}(p,q)+\boldsymbol{\beta}(q',p')\bigr]B_1,
\end{equation}
where $$
\boldsymbol{\beta}(p,q)=\frac{2^{1/q+1}}{(2^{(r-q)/p}-1)^{1/r}(2^{q/r}-1)^{1/p}}.
$$
Observe that
$$\lim_{q\uparrow p} \bigl[\boldsymbol{\beta}(p,q)+\boldsymbol{\beta}(q',p')\bigr]=\infty.
$$
In the rest cases, the following inequalities take place for $ q <p $: \begin{gather}\label{3B}\Bigl(\frac{q}{r}\Bigr)^{\frac{1}{q}}\Bigl(\frac{p'}{2r}\Bigr)^{\frac{1}{p'}}B_1\le C_2\le \begin{cases}\mathbf{C}_{1,\beta'}\,\bigl[B_1+B_2\bigr]\le \mathbf{C}_{1,\beta'}\bigl[1+\boldsymbol{\beta}(p,q)\bigr]B_1, & \frac{r}{p}\ge 1\ {\&}\ \frac{r}{q'}< 1,\\
\mathbf{C}_{\beta,1}\,\bigl[B_1+B_3\bigr]\le \mathbf{C}_{\beta,1}\bigl[1+\boldsymbol{\beta}(q',p')\bigr]B_1, & \frac{r}{p}< 1\ {\&}\ \frac{r}{q'}\ge 1,\\
\mathbf{C}_{\beta,\beta'}\,B_1\, & \frac{r}{p}< 1\ {\&}\ \frac{r}{q'}< 1.\end{cases}\end{gather}
On the strength of the restrictions on the parameters $ p $ and $ q $, all coefficients in \eqref{3B} are finite. In the first zone $ r \to \infty $ only if $ p, q \to \infty $; similarly, in the second zone $ r \to \infty $ only if $ p, q \to 1 $; and in the third zone $ r $ cannot approach $ \infty $. In addition, $ \mathbf{C}_{1,1} $ in \eqref{BBB} does not diverge for $ q \uparrow p $, and, therefore, the second inequality gives an upper bound in Sawyer's theorem for $ p = q $, since $ \lim \limits_ {q \uparrow p} B_i = A_i, i = 1,2,3 $ (see \eqref{AB}).

The upper estimates in \eqref {BBB}--\eqref{3B} can be proven similarly to the upper bound for $ C_2 $ in the case $ q <p $ in the main theorem. The only difference is that for $ r / p \ge 1 $, instead of Lemma \ref{lem1}, one should use the inequality
\begin{multline*}
\mathbf{V}_{(a,b)\times(c,d)}\le\bigl[I_2 \sigma(b,d)\bigr]^{\frac{q}{p}}
\biggl[\int_a^b\int_c^d\frac{\chi_{\supp\,w}(x,y)}{\bigl[I_2 \sigma(x,y)\bigr]^{\frac{r}{p}}}\,d_x\,d_y\biggl(\int_0^{x}\int_0^{y}(I_2\sigma)^q w\biggr)^{\frac{r}{q}}\biggr]^{\frac{q}{r}}.
\end{multline*} Similarly, for $ r / q'\ge 1 $, instead of Lemma \ref{lem2}, the following estimate should be applied:
\begin{multline*}
\mathbf{W}_{(a,b)\times(c,d)}\le\bigl[I_2^\ast w(a,c)\bigr]^{\frac{p'}{q'}}
\biggl[\int_a^b\int_c^d\frac{\chi_{\supp\,\sigma}(x,y)}{\bigl[I_2^\ast w(x,y)\bigr]^{\frac{r}{q'}}}\,d_x\,d_y\biggl(\int_{x}^\infty\int_{y}^\infty(I_2^\ast w)^{p'} \sigma\biggr)^{\frac{r}{p'}}\biggr]^{\frac{p'}{r}}.
\end{multline*}

To establish $ B_2 \le \boldsymbol {\beta} (p, q) B_1 $ we split $ {\mathbb{R}_+^2} $ into domains $ \omega_k $ (as in Lemma \ref{lem1}). Then
\begin{align*}
&\int_{\mathbb{R}_+^2}\bigl[I_2 \sigma(x,y)\bigr]^{-\frac{r}{p}}\,d_x\,d_y\biggl(\int_0^{x}\int_0^{y}(I_2\sigma)^q w\biggr)^{\frac{r}{q}}\\&=
\sum_{k\le K_\sigma}\int_{\omega_k\setminus\omega_{k+1}}\bigl[I_2 \sigma(x,y)\bigr]^{-\frac{r}{p}}\,d_x\,d_y\biggl(\int_0^{x}\int_0^{y}(I_2\sigma)^q w\biggr)^{\frac{r}{q}}\end{align*}\begin{align*}&\le
\sum_{k\le K_\sigma}2^{-kr/p}\int_{\omega_k\setminus\omega_{k+1}}d_x\,d_y\biggl(\int_0^{x}\int_0^{y}(I_2\sigma)^q w\biggr)^{\frac{r}{q}}\\&\le
\sum_{k\le K_\sigma}2^{-kr/p}\int_{\mathbb{R}_+^2\setminus\omega_{k+1}}d_x\,d_y\biggl(\int_0^{x}\int_0^{y}(I_2\sigma)^q w\biggr)^{\frac{r}{q}}.
\end{align*} Since
\begin{multline*}
\int_{\mathbb{R}_+^2\setminus\omega_{k+1}}d_x\,d_y\biggl(\int_0^{x}\int_0^{y}(I_2\sigma)^q w\biggr)^{\frac{r}{q}}=
\int_{\mathbb{R}_+^2}\chi_{\mathbb{R}_+^2\setminus\omega_{k+1}}(x,y)\,\,d_x\,d_y\biggl(\int_0^{x}\int_0^{y}(I_2\sigma)^q w\biggr)^{\frac{r}{q}}\\=\int_{\mathbb{R}_+^2}d_x\,d_y\biggl(\int_0^{x}\int_0^{y}\chi_{\mathbb{R}_+^2\setminus\omega_{k+1}}(I_2\sigma)^q w\biggr)^{\frac{r}{q}}=
\biggl(\int_{\mathbb{R}_+^2\setminus\omega_{k+1}}(I_2\sigma)^q w\biggr)^{\frac{r}{q}},
\end{multline*} then we have
$$
\sum_{k\le K_\sigma}2^{-kr/p}\int_{\mathbb{R}_+^2\setminus\omega_{k+1}}d_x\,d_y\biggl(\int_0^{x}\int_0^{y}(I_2\sigma)^q w\biggr)^{\frac{r}{q}}\\ =
\sum_{k\le K_\sigma}2^{-kr/p}\biggl(\int_{\mathbb{R}_+^2\setminus\omega_{k+1}}(I_2\sigma)^q w\biggr)^{\frac{r}{q}}.
$$ From Proposition \ref{propGHS}(b) with $ \tau = 2^{\frac{q}{p}} $ and $ \gamma = r / q $
\begin{multline*}
\sum_{k\le K_\sigma}2^{-kr/p}\biggl(\int_{\mathbb{R}_+^2\setminus\omega_{k+1}}(I_2\sigma)^q w\biggr)^{\frac{r}{q}}=\sum_{k\le K_\sigma}2^{-kr/p}\biggl(\sum_{m\le k}\int_{\omega_m\setminus\omega_{m+1}}(I_2\sigma)^q w\biggr)^{\frac{r}{q}}\\\le\frac{2^{r/p}}{(2^{(r-q)/p}-1)(2^{q/r}-1)^{r/p}}\sum_{k\le K_\sigma}2^{-kr/p}\biggl(\int_{\omega_k\setminus\omega_{k+1}}(I_2\sigma)^q w\biggr)^{\frac{r}{q}}\\\le
\frac{2^{r/p+r}}{(2^{(r-q)/p}-1)(2^{q/r}-1)^{r/p}}\sum_{k\le K_\sigma}2^{kr/p'}\biggl(\int_{\omega_k\setminus\omega_{k+1}} w\biggr)^{\frac{r}{q}}.
\end{multline*} By analogy with the proof of Lemma \ref{lem1}, we can write
$$|\omega_k\setminus\omega_{k+1}|_w^{\frac{r}{q}}\le|\omega_k|_w^{\frac{r}{q}}=\int_{\mathbb{R}^2_+}d_xd_y\bigl[I^\ast_2(\chi_{\omega_k}w)(x,y)\bigr]^{\frac{r}{q}}= \int_{\omega_k}d_xd_y\bigl[I^\ast_2w(x,y)\bigr]^{\frac{r}{q}}.$$ Hence (see Proposition \ref{propGHS}(a)),
\begin{multline*}
\sum_{k\le K_\sigma}2^{kr/p'}\biggl(\int_{\omega_k\setminus\omega_{k+1}} w\biggr)^{\frac{r}{q}}\le \sum_{k\le K_\sigma}2^{kr/p'}\int_{\omega_k}d_xd_y\bigl[I^\ast_2w(x,y)\bigr]^{\frac{r}{q}}\\\le\sum_{k\le K_\sigma}2^{k}\int_{\omega_k}\bigl[I_2\sigma(x,y)\bigr]^{\frac{r}{q'}}d_xd_y\bigl[I^\ast_2w(x,y)\bigr]^{\frac{r}{q}}\\\le  2\sum_{k\le K_\sigma}
\int_{\omega_k\setminus\omega_{k+1}}\bigl[I_2\sigma(x,y)\bigr]^{\frac{r}{p'}}d_xd_y\bigl[I^\ast_2w(x,y)\bigr]^{\frac{r}{q}}= B_1^r.
\end{multline*}

Similarly, one can show that $ B_3 \le \boldsymbol {\beta} (q', p') B_1 $. Thus, \eqref{BBB} and \eqref {3B} are valid.

\section{Sufficient condition}

The one--dimensional analog of the condition \eqref{n3'} is the boundedness of the Muckenhoupt constant \cite{M}, of the condition \eqref{n4'} --- the boundedness of the Tomaselli functional \cite[definition (11)]{Tom}, and the analogs of the constants $ B_1, \, B_2 $ are the Maz'ya--Rozin \cite[\S \, 1.3.2]{Maz} and Persson--Stepanov \cite[Theorem 3]{PS} functionals, respectively. The constants have been generalized to the scales of equivalent conditions in \cite{PSW} (see also \cite{GKPW} for the case $p\le q$). In the following theorem we find a sufficient condition for the inequality \eqref{aux} to hold, having the form \eqref{Bv}, where $B_v$ is a two--dimensional analog of the constant $ \mathcal{B}_{MR}^{(1)} (1 / r ) $ from \cite{PSW} in the one--dimensional case.

\begin{theorem} Let $1<q<p<\infty$. The inequality \eqref{aux} holds if \begin{equation}\label{Bv}{B}_v:=\biggl(\int_{\mathbb{R}^2_+} \sigma(u,z)
\biggl(\int_u^\infty \int_z^\infty
(I_2 \sigma)^{q-1}w\biggr)
^{\frac{r}{q}}\,{d}u\,{d}z\biggr)^{\frac{1}{r}}<\infty,
\end{equation}
where $C_2\lesssim {B}_v$.
\end{theorem}

\begin{proof}[Proof] We apply Sawyer's scheme of partitioning $ \mathbb{R}_+^2 $ into rectangles from the proof of the sufficiency in Theorem \ref{aga}. Compared to Figure 1, Figure 2 below has a rectangle $Q_j^k=(0,x_j^k)\times(0,y_j^k)$ added.

{\setlength{\unitlength}{0.0105in}{
\begin{picture}(350,350)
\put(10,10){\vector(0,1){320}}
\put(10,10){\vector(1,0){380}}

{\linethickness{.36mm}
\qbezier(40,270)(80,45)(300,25)
\qbezier(70,320)(130,60)(380,35)}

\put(60,0){\small $x_1^k$}
\put(120,0){\small $x_2^k$}
\put(200,0){\small $x_3^k$}

\put(-3,43){\small $y_3^k$}
\put(-3,94){\small $y_2^k$}
\put(-3,182){\small $y_1^k$}

\put(168,150){$\widetilde{\bf S}_{\boldsymbol{2}}^{\boldsymbol{k}}$}
\put(170,73){$\widetilde{\bf R}_{\boldsymbol{2}}^{\boldsymbol{k}}$}
\put(280,180){${\bf T}_{\boldsymbol{2}}^{\boldsymbol{k}}$}
\put(65,60){${\bf Q}_{\boldsymbol{2}}^{\boldsymbol{k}}$}

{\linethickness{.0001mm}
\put(124,47){\line(1,1){50}}\put(132,47){\line(1,1){50}}\put(140,47){\line(1,1){50}}\put(148,47){\line(1,1){50}}\put(156,47){\line(1,1){48}}\put(164,47){\line(1,1){40}}
\put(172,47){\line(1,1){32}}\put(180,47){\line(1,1){25}}
\put(188,47){\line(1,1){16}}\put(195,47){\line(1,1){10}}

\put(124,55){\line(1,1){42}}\put(124,63){\line(1,1){34}}\put(124,71){\line(1,1){26}}\put(124,79){\line(1,1){18}}\put(124,87){\line(1,1){10}} }

{\linethickness{.0001mm}
\put(122,184){\line(1,-1){80}}\put(130,184){\line(1,-1){72}}\put(138,184){\line(1,-1){64}}\put(146,184){\line(1,-1){56}}\put(154,184){\line(1,-1){48}}\put(162,184){\line(1,-1){40}}
\put(170,184){\line(1,-1){32}}\put(178,184){\line(1,-1){25}}
\put(186,184){\line(1,-1){16}}\put(193,184){\line(1,-1){10}}

\put(124,174){\line(1,-1){77}}\put(124,166){\line(1,-1){68}}\put(124,158){\line(1,-1){60}}\put(124,150){\line(1,-1){52}}\put(124,142){\line(1,-1){44}}\put(124,134){\line(1,-1){36}}
\put(124,126){\line(1,-1){28}}\put(124,118){\line(1,-1){20}}
\put(124,110){\line(1,-1){12}} }

{\linethickness{.0001mm}
\put(205,107){\line(1,-1){10}}\put(205,117){\line(1,-1){20}}\put(205,127){\line(1,-1){30}}\put(205,137){\line(1,-1){40}}\put(205,147){\line(1,-1){50}}\put(205,157){\line(1,-1){60}}
\put(205,167){\line(1,-1){70}}\put(205,177){\line(1,-1){80}}
\put(205,187){\line(1,-1){90}}\put(205,197){\line(1,-1){100}}
\put(205,207){\line(1,-1){110}}\put(205,217){\line(1,-1){120}}\put(205,227){\line(1,-1){130}}\put(205,237){\line(1,-1){140}}
\put(205,247){\line(1,-1){150}}\put(205,257){\line(1,-1){160}}
\put(205,267){\line(1,-1){170}}\put(205,277){\line(1,-1){170}}\put(205,287){\line(1,-1){170}}\put(205,297){\line(1,-1){170}}
\put(205,307){\line(1,-1){170}}\put(205,317){\line(1,-1){170}}
\put(205,327){\line(1,-1){170}}
}

{\linethickness{.0001mm}
\put(12,10){\line(1,1){86}}\put(22,10){\line(1,1){86}}
\put(32,10){\line(1,1){86}}\put(42,10){\line(1,1){82}}
\put(52,10){\line(1,1){72}}\put(62,10){\line(1,1){62}}
\put(72,10){\line(1,1){52}}\put(82,10){\line(1,1){42}}
\put(92,10){\line(1,1){32}}\put(102,10){\line(1,1){22}}
\put(112,10){\line(1,1){12}}

\put(10,18){\line(1,1){78}}\put(10,28){\line(1,1){68}}
\put(10,38){\line(1,1){58}}\put(10,48){\line(1,1){48}}
\put(10,58){\line(1,1){38}}\put(10,68){\line(1,1){28}}
\put(10,77){\line(1,1){19}}\put(10,86){\line(1,1){10}}
}
\put(65,97){\line(0,1){230}}
\multiput(65,10)(0,8){11}{\line(0,1){5}}

\put(10,184){\line(1,0){114}}
\multiput(124,184)(8.5,0){10}{\line(1,0){5}}

\put(124,10){\line(0,1){175}}
\multiput(124,185)(0,8){18}{\line(0,1){5}}

\put(10,97.3){\line(1,0){198}}

\put(204,10){\line(0,1){174}}
\multiput(204,184)(0,8){18}{\line(0,1){5}}

\put(124,47){\line(1,0){190}}

\multiput(310,47)(8.5,0){7}{\line(1,0){5}}
\multiput(10,47)(8.5,0){17}{\line(1,0){5}}

\multiput(210,97.3)(8.5,0){20}{\line(1,0){5}}

\put(385,38){\small $\partial\Omega_{k+1}$}
\put(310,25){\small $\partial\Omega_{k}$}

\put(203,-15){\scriptsize\rm Fig. 2}
\end{picture}}}\vspace{8mm}

Denote $\widetilde{E}_j^k:=E_j^k\cup \bigl(\widetilde{S}_j^k\cap(\Omega_{k+2}-\Omega_{k+3})\bigr)$. Then (see \eqref{20})
\begin{eqnarray}\label{20r}
\int_{\mathbb{R}^2_+} (I_2f)^qw\approx\sum_{k,j}
3^{kq}\big|\widetilde{E}_j^k\big|_w.\end{eqnarray}

Put $ g \sigma: = f $ and write
\begin{equation}\label{100r} \sum_{k,j}
3^{kq}\big|\widetilde{E}_j^k\big|_w=
\sum_{k,j} \bigl|\widetilde{E}_j^k\bigr|_w\biggl(\iint_{Q_j^k}f\biggr)^q=\sum_{k,j}
\bigl|\widetilde{E}_j^k\bigr|_w\bigl|Q_j^k\bigr|_{\sigma}^q\biggl(\frac{1}{\bigl|Q_j^k\bigr|_{\sigma}}
\iint_{Q_j^k}g\sigma\biggr)^q.\end{equation} For an integer $ l $ by $ \Gamma_l $ we denote the set of pairs $ (k, j) $ such that
$\bigl|\widetilde{E}_j^k\bigr|_w>0$ and
\begin{eqnarray*}\label{110r}
2^l<\frac{1}{\bigl|Q_j^k\bigr|_{\sigma}}
\iint_{Q_j^k}g\sigma\leq 2^{l+1},\qquad(k,j)\in\Gamma_l.
\end{eqnarray*}
By analogy with how it was done in the proof of \cite[Theorem 1A]{Saw1}, we show that \begin{eqnarray*}\label{120'r}
2^{l-1}<\frac{1}{\bigl|Q_j^k\bigr|_{\sigma}}
\iint_{Q_j^k}g \sigma \chi_{\{g>2^{l-1}\}},\qquad \textrm{for}\
\textrm{all}\ \ j,k.\end{eqnarray*} Indeed, this follows from the fact that
\begin{eqnarray*}2^{l}<
\frac{1}{\bigl|Q_j^k\bigr|_{\sigma}}
\iint_{Q_j^k}g\sigma=
\frac{1}{\bigl|Q_j^k\bigr|_{\sigma}}\biggl[
\iint_{Q_j^k\cap\{g>2^{l-1}\}}g\sigma +
\iint_{Q_j^k\cap\{g\le 2^{l-1}\}}g\sigma\biggr]\\
\le\frac{1}{\bigl|Q_j^k\bigr|_{\sigma}}
\iint_{Q_j^k\cap\{g>2^{l-1}\}}g\sigma+2^{l-1}.
\end{eqnarray*}
Further, we write for fixed $ l $:
\begin{align*}
\sum_{(k,j)\in\Gamma_l} \bigl|\widetilde{E}_j^k\bigr|_w\bigl|Q_j^k\bigr|_{\sigma}^q\overset{\eqref{120'r}}\lesssim &
2^{-l}\sum_{(k,j)\in\Gamma_l}\bigl|\widetilde{E}_j^k\bigr|_w\bigl|Q_j^k\bigr|_{\sigma}^{q-1}\iint_{Q_j^k}g\sigma \chi_{\{g>2^{l-1}\}}\nonumber\\
\le& 2^{-l}\sum_{(k,j)\in\Gamma_l}\int_{\widetilde{E}_j^k}w(x,y)\bigl[I_2\sigma(x,y)\bigr]^{q-1}\biggl(\int_0^x\int_0^y
g\sigma \chi_{\{g>2^{l-1}\}}\biggr){d}x\,{d}y.\end{align*}
Combining the last estimate and \eqref{100r}, we obtain
\begin{align*}\sum_{k,j} 3^{kq}\bigl|\widetilde{E}_j^k\bigr|_w&\lesssim  \sum_l 2^{lq}\sum_{(k,j)\in\Gamma_l} \bigl|\widetilde{E}_j^k\bigr|_w\bigl|
Q_j^k\bigr|_{\sigma}^q\\
&\le\sum_l 2^{l(q-1)}\sum_{(k,j)\in\Gamma_l}\int_{\widetilde{E}_j^k}w(x,y)\bigl[I_2\sigma(x,y)\bigr]^{q-1}\biggl(\int_0^x\int_0^y
g\sigma \chi_{\{g>2^{l-1}\}}\biggr){d}x\,{d}y \nonumber\\&=\sum_{k,j}
\int_{\widetilde{E}_j^k}w(x,y)\bigl[I_2\sigma(x,y)\bigr]^{q-1}\biggl(\int_0^x\int_0^y
g\sigma \Bigl[\sum_l 2^{l(q-1)}\chi_{\{g>2^{l-1}\}}\Bigr]\biggr){d}x\,{d}y. \nonumber
\end{align*} Since $2^{l_0-1}<g(s,t)\le 2^{l_0}$ almost everywhere for fixed $(s,t)$ then $g(s,t)>2^{l-1}$ for $l\le l_0$ and, therefore,
$$\sum_l 2^{l(q-1)}\chi_{\{g>2^{l-1}\}}=\sum_{l\le l_0} 2^{l(q-1)}=2^{l_0(q-1)}\sum_{l\le l_0} 2^{(l-l_0)(q-1)}
\approx 2^{l_0(q-1)}.$$
From this and H\"older's inequalities with exponents $ p / q $ and $ r / q $, we find that
\begin{align}\label{aux'''}\sum_{k,j} 3^{kq}\bigl|\widetilde{E}_j^k\bigr|_w\lesssim&\sum_{k,j}
\int_{{E}_j^k}w(x,y)\bigl[I_2\sigma(x,y)\bigr]^{q-1}\biggl(\int_0^x\int_0^y g^q(s,t)\sigma(s,t)\,{d}s\,{d}t
\biggr){d}x\,{d}y\nonumber\\
=&\int_{\mathbb{R}^2_+}
w(x,y)\bigl[I_2\sigma(x,y)\bigr]^{q-1}\biggl(\int_0^x\int_0^y
g^q(s,t)\sigma(s,t)\,{d}s\,{d}t\biggr)
{d}x\,{d}y\nonumber\\
=&\int_{\mathbb{R}^2_+}g^q(s,t)\sigma(s,t)\biggl(\int_s^\infty\int_t^\infty w(x,y)
\bigl[I_2\sigma(x,y)\bigr]^{q-1}\,{d}x\,{d}y\biggr) {d}s\,{d}t\nonumber\end{align} \begin{align}
\le&\biggl(\int_{\mathbb{R}^2_+}g^p\sigma\biggr)^{\frac{q}{p}}
\biggl(\int_{\mathbb{R}^2_+}\sigma(s,t)\biggl(\int_s^\infty\int_t^\infty(I_2\sigma)^{q-1}w\biggr)^{\frac{r}{q}}\, {d}s\,{d}t\biggr)^{\frac{q}{r}}\nonumber\\=&{B}_v^q\biggl(\int_{\mathbb{R}^2_+}g^p\sigma\biggr)^{\frac{q}{p}},\end{align} since the sets $\widetilde{E}_j^k$ are disjoint and $g^p\sigma=f^pv$.
The estimates \eqref{20r} and \eqref{aux'''} imply the validity of \eqref{aux} for all $ f $ from the subclass $ M.$\qed \end{proof}

There is also a dual statement of the last theorem with the functional
  \[{B}_w:=\biggl(\int_{\mathbb{R}^2_+} w(u,z)
\biggl(\int_0^u \int_0^z
(I_2^\ast w)^{p'-1}\sigma\biggr)
^{\frac{r}{p'}}\,{d}u\,{d}z\biggr)^{\frac{1}{r}}\] instead of $B_v$. The proof of this fact is similar and can be carried out through the operator $I_2^\ast f$.

If the weights $ v $ and $ w $ are factorizable, then the condition $ B _v <\infty $ (or $ B_w <\infty $) is necessary and sufficient for the \eqref{aux} to be true
in the case of $ 1 <q <p <\infty $, moreover $ C_2 \approx {B} _v \approx B_w $.

\section{Multidimensional case with factorizable weights}
It was established by A. Wedestig in \cite{W1} (see also \cite{B5}) for the case $ n = 2 $ that if the weight function $ v $ in \eqref{n2}
 is factorizable, that is, $ v (x_1, x_2) = v_1 (x_1) v_2 (x_2) $, then it is possible to characterize the inequality \eqref{n2} by only
one functional for all $1<p\le q<\infty$.

\begin{theorem}{\rm\cite[Theorem 1.1]{B5}}\label{T2-mult}
Let $n=2,$ $1<p\le q<\infty,$ $s_1,s_2\in (1,p)$ and
$v(x_1,x_2)=v_1(x_1)v_2(x_2).$ Then the inequality \eqref{n2}
holds for all $f\ge 0$ if and only if
\begin{multline*}A_W(s_1,s_2)\colon=\sup_{(t_1,t_2)\in\mathbb{R}_+^2}\bigl[I_1\sigma_1(t_1)\bigr]
^{\frac{s_1-1}{p}}\bigl[I_1\sigma_2(t_2)\bigr]^{\frac{s_2-1}{p}}\\\times\biggl(\int_{t_1}^\infty\int_{t_2}^\infty
\bigl(I_1\sigma_1\bigr)^{\frac{q(p-s_1)}{p}}\bigl(I_1\sigma_2\bigr)^{\frac{q(p-s_2)}{p}}w\,
\biggr)^{\frac{1}{q}}<\infty,
\end{multline*} where $\sigma_i:=v_i^{1-p'}$, $i=1,2$. Moreover, $C_2\approx A_W(s_1,s_2)$ with equivalence constants dependent on parameters $p$, $q$ and $s_1$, $s_2$ only.\end{theorem}
The result of this theorem can be generalized to $ n> 2. $

A number of statements similar to \cite[Theorem 1.1]{B5} were obtained in \cite{PU} under the condition that weight functions $ v $ or $ w $ satisfy
\begin{equation}\label{n8}
v(y_1,\ldots,y_n)=v_1(y_1)\ldots
v_n(y_n)\end{equation} or
\begin{equation}\label{n8'}
w(x_1,\ldots,x_n)=w_1(x_1)\ldots
w_n(x_n).\end{equation}
\begin{theorem}\label{T4}{\rm\cite[Theorems 2.1, 2.2]{PU}}
Let $1<p\le q<\infty$ and the weight function $v$ satisfy the condition \eqref{n8}. Then the inequality \eqref{n2} holds for all $f\ge 0$ \\ {\rm (i)} if and only if $A_{M_n}<\infty$, where
\begin{equation*}A_{M_n}\colon=\sup_{(t_1,\ldots,t_n)\in\mathbb{R}_+^n}\bigl[I_n^\ast w(t_1,\ldots,t_n)\bigr]^{\frac{1}{q}}
\bigl[I_1\sigma_1(t_1)\bigr]^{\frac{1}{p'}}\ldots
\bigl[I_1\sigma_n(t_n)\bigr]^{\frac{1}{p'}};\end{equation*} {\rm (ii)}
if and only if $A_{T_n}<\infty,$ where
\begin{equation*}A_{T_n}=\sup_{(t_1,\ldots,t_n)\in\mathbb{R}_+^n}
\bigl[I_1\sigma_1(t_1)\bigr]^{-\frac{1}{p}}\ldots
\bigl[I_1\sigma_n(t_n)\bigr]^{-\frac{1}{p}}\biggl(\int_0^{t_1}\ldots\int_0^{t_n}
\bigl(I_1\sigma_1\bigr)^q\ldots\bigl(I_1\sigma_n\bigr)^qw\biggr)^{\frac{1}{q}}.\end{equation*}
Besides, $C_n\approx A_{M_n}\approx  A_{T_n}$ with equivalence constants depending on $ p, $ $ q $ and $n.$\end{theorem}
\begin{theorem}\label{T2d}{\rm\cite[Theorems 2.4, 2.5]{PU}}
Let $ 1 <p \le q <\infty $ and the weight $ w $ satisfy the condition \eqref{n8'}. Then the inequality \eqref{n2} is true \\ {\rm (i)} if and only if $A_{M_n}^\ast<\infty$, where with $\sigma:=v^{1-p'}$
\begin{equation*}A_{M_n}^\ast\colon=\sup_{(t_1,\ldots,t_n)\in\mathbb{R}_+^n}
\bigl[I_n\sigma(t_1,\ldots,t_n)\bigl]^{\frac{1}{p'}}\bigl[I_1^\ast w_1(t_1)\bigr]^{\frac{1}{q}}\ldots
\bigl[I_1^\ast w_n(t_n)\bigr]^{\frac{1}{q}};\end{equation*} {\rm (ii)}
if and only if $A_{T_n}^\ast<\infty,$ where
\begin{equation*}A_{T_n}^\ast=\sup_{(t_1,\ldots,t_n)\in\mathbb{R}_+^n}
\bigl[I_1^\ast w_1(t_1)\bigr]^{-\frac{1}{q'}}\ldots
\bigl[I_1^\ast w_n(t_n)\bigr]^{-\frac{1}{q'}}
\biggl(\int_{t_1}^\infty\ldots\int_{t_n}^\infty
 \bigl(I_1^\ast w_1\bigr)^{p'}\ldots
\bigl(I_1^\ast w_n\bigr)^{p'}\sigma\biggr)^{\frac{1}{p'}}.\end{equation*} Besides, $C_n\approx A_{M_n}^\ast\approx A_{T_n}^\ast$ with equivalence constants depending on $ p, $ $ q $ and $n.$\end{theorem}

\begin{theorem}\label{Tq1}{\rm\cite[Theorems 3.1, 3.2]{PU}}
Let $1<q<p<\infty$. Suppose that the weight function $v$ in \eqref{n2} satisfies the condition \eqref{n8} and $I_1\sigma_1(\infty)=\ldots=I_1\sigma_n(\infty)=\infty.$
Then \eqref{n2} is valid for all $f\ge 0$ on $\mathbb{R}^n_+$ with $C_n<\infty$ independent of functions $f$
\\ {\rm (i)} if and only if $B_{MR_n}<\infty$, where
\begin{equation*}B_{MR_n}:=
\biggl(\int_{\mathbb{R}^n_+} \bigl[I_n^\ast w(t_1,\ldots,t_n)\bigr]^{\frac{r}{q}}
\bigl[I_1\sigma_1(t_1)\bigr]^{\frac{r}{q'}}\sigma_1(t_1)\ldots \bigl[I_1\sigma_n(t_n)\bigr]^{\frac{r}{q'}}\sigma_n(t_n)\,{d}t_1\ldots \,{d}t_n\biggr)^{\frac{1}{r}};\end{equation*} {\rm (ii)}
if and only if $B_{PS_n}<\infty$, where
\begin{multline*}B_{PS_n}:=
\biggl(\int_{\mathbb{R}^n_+}\biggl(\int_0^{t_1}\ldots\int_0^{t_n}
\bigl[I_1\sigma_1(t_1)\bigr]^q\ldots \bigl[I_1\sigma_n(t_n)\bigr]^qw(x_1,\ldots, x_n)\,{d}x_1\ldots\,dx_n\biggr)
^{\frac{r}{q}}\biggr.\\\quad\times\biggl.\bigl[I_1\sigma_1(t_1)\bigr]^{-\frac{r}{q}}\sigma_1(t_1)\ldots
\bigl[I_1\sigma_n(t_n)\bigr]^{-\frac{r}{q}}\sigma_n(t_n) \,{d}t_1\ldots\,dt_n\biggr)^{\frac{1}{r}}.\end{multline*}
Moreover, $C_n\approx B_{MR_n}\approx B_{PS_n}$ with equivalence constants dependent on $ p, $ $ q $ and $n.$
\end{theorem}
\begin{theorem}\label{Tq3}{\rm\cite[Theorems 3.3, 3.4]{PU}}
Let $1<q<p<\infty$. Assume that $w$ in \eqref{n2} satisfies \eqref{n8'} and $I_1^\ast w_1 (0)=\ldots=I_1^\ast w_n(0)=\infty$. Then
\eqref{n2} is valid for all $f\ge 0$ on $\mathbb{R}^n_+$ with $C_n<\infty$ independent of functions $f$\\ {\rm (i)} if and only if $B_{MR_n}^\ast<\infty$, where
\begin{equation*}\label{AMn}B_{MR_n}^\ast:=
\biggl(\int_{\mathbb{R}^n_+}\bigl[I_n\sigma(t_1,\ldots,t_n)\bigr]^{\frac{r}{p'}}\bigl[I_1^\ast w_1(t_1)\bigr]^{\frac{r}{p}}w_1(t_1)\ldots \bigl[I_1^\ast w_n(t_n)\bigr]^{\frac{r}{p}}w_n(t_n)\,{d}t_1\ldots\,{d}t_n\biggr)^{\frac{1}{r}};\end{equation*} {\rm (ii)}
if and only if $B_{PS_n}^\ast<\infty$, where
\begin{multline*}B_{PS_n}^\ast:=
\biggl(\int_{\mathbb{R}^n_+}\biggl(\int_{t_1}^\infty\ldots\int_{t_n}^\infty
\bigl(I_1^\ast w_1\bigr)^{p'}\ldots\bigl(I_1^\ast w_n\bigr)^{p'}
\sigma\biggr)
^{\frac{r}{p'}}\biggr.\\\quad\times\biggl.\bigl[I_1^\ast w_1(t_1)\bigr]^{-\frac{r}{p'}}w_1(t_1)\ldots
\bigl[I_1^\ast w_n(t_n)\bigr]^{-\frac{r}{p'}}w_n(t_n)\,{d}t_1\ldots
\,{d}t_n\biggr)^{\frac{1}{r}}.\end{multline*} Moreover, $C_n\approx B_{MR_n}^\ast\approx B_{PS_n}^\ast$ with equivalence constants dependent on $ p, $ $ q $ and $n.$
\end{theorem}

\renewcommand{\refname}{Bibliography}


\begin{thebibliography}{100}
\bibitem{Barza}  Barza~S. Weighted multidimensional integral inequalities and
applications: Doctoral Thesis. --- Lule\aa~ University of Technology, Department of Mathematics. --- 1999. --- N 1999:30. --- 134 pp.

\bibitem{GKPW} Gogatishvili~A., Kufner~A., L.--E.~Persson, Wedestig~A. An equivalence theorem for integral conditions related to Hardy's inequality //
Real Anal. Exchange --- 2003/04. --- Vol.~29, N~2. --- P.~867--880.

\bibitem{GHS} Goldman~M.L., Heinig~H.P., Stepanov~V.D. On the principle of duality in Lorentz spaces // Can. J. Math. --- 1996. --- Vol. 48. --- P. 959--979.\\[-6mm]

\bibitem{KMP1} Kokilashvili~V., Meskhi~A., Persson~L.--E. Weighted norm inequalities for integral transforms with product kernels. --- Nova Science Publishers,
New--York, 2009.\\[-6mm]

\bibitem{KMP} Kufner~A., Maligranda~L., Persson~L.--E. The Hardy inequality. About its history and some related results. --- Vydavatelsk\'y Servis, Plze\v n,
2007.\\[-6mm]


\bibitem{KPS}   Kufner А., Persson L.-E., Samko N. Weighted inequalities of Hardy-type. World Scientific Publishing Co. Inc. New Jersey, 2017.\\[-6mm]

\bibitem{Maz}  Maz'ja V. G. Sobolev spaces. --- Berlin: Springer Series in Soviet Mathematics. Springer-Verlag, 1985.\\[-6mm]

\bibitem{Mes} Meskhi A. A note on two-weight inequalities for multiple Hardy-type operators // J. Funct. Spaces Appl. --- 2005. --- Vol. 3, N 3. --- P. 223--237.\\[-6mm]


\bibitem {M} Muckenhoupt B. Hardy inequalities with weights //
Studia Math. --- 1972. --- Vol. 44. --- P. 31--38.\\[-6mm]

\bibitem {PS}
Persson, L.-E., Stepanov, V. D. Weighted integral inequalities
with the geometric mean operator // {J. Inequal. Appl.} --- 2002. --- Vol. 7. --- P. 727--746.\\[-6mm]

\bibitem{PSW} Persson~L.--E, Stepanov~V., Wall~P. Some scales of equivalent weight characterizations of Hardy's inequality: the case q<p // Math. Inequal. Appl. --- 2007. --- Vol.~10, N~2. ---P.~267--279.\\[-6mm]

\bibitem{PU} Persson L.–E., Ushakova E.P. Some multi--dimensional Hardy type integral
 inequalities // J. Math. Inequal. --- 2007. --- Vol. 1, N 3. --- 301--319.\\[-6mm]

\bibitem{PSU} Prokhorov D. V., Stepanov V. D. and Ushakova E. P. Hardy--Steklov integral operators: Part I. // Proc. Steklov Inst. Math. --- 2016. --- Vol.~300, N~2. --- P.~1--112.\\[-6mm]

\bibitem{Saw1}  Sawyer~E. Weighted inequalities for two--dimensional Hardy
operator // Studia Math. --- 1985. --- Vol.~82, N~1. --- P.~1--16.\\[-6mm]

\bibitem{Tom} Tomaselli, G. A class of inequalities //
{Boll. Unione Mat. Ital}. --- 1969. --- Vol. {2}. ---  622--631.\\[-6mm]

\bibitem{W1}  Wedestig~A. Weighted inequalities of Hardy--type and
their limiting inequalities: Doctoral Thesis. --- Lule\aa~ University of Technology, Department of Mathematics. --- 2003. --- N 2003:17. --- 106 pp.

\bibitem{B5}  Wedestig~A. Weighted inequalities for the Sawyer two--dimensional
Hardy operator and its limiting geometric mean operator // J. Inequal. Appl. --- 2005. --- Vol.~4. --- P.~387--394.

\end{thebibliography}
\end{document}